\documentclass[preprint,12pt]{elsarticle}

\usepackage[utf8]{inputenc}
\usepackage[T1]{fontenc}

\usepackage[english]{babel}

\usepackage{amsmath}
\usepackage{amsfonts}
\usepackage{amssymb}
\usepackage{amsthm}
\usepackage{natbib}

\usepackage{a4wide}
\newtheorem{thm}{Theorem}[section]
\newtheorem{lem}[thm]{Lemma}

\usepackage{amssymb}
\journal{}

\begin{document}

\begin{frontmatter}



\title{Maximum likelihood estimation for a bivariate Gaussian process under fixed domain asymptotics}


\author{Velandia Daira$^{a,c}$\corref{}}
\ead{dvelandia@unitecnologica.edu.co}
\author{Bachoc François$^b$\corref{}}
\ead{francois.bachoc@math.univ-toulouse.fr} 
\author{Bevilacqua Moreno$^c$ \corref{cor1}}
\ead{moreno.bevilacqua@uv.cl}
\author{Gendre Xavier$^b$ \corref{}}
\ead{xavier.gendre@math.univ-toulouse.fr} 
\author{Loubes Jean-Michel$^b$ \corref{}}
\ead{loubes@math.univ-toulouse.fr}

\cortext[cor1]{moreno.bevilacqua@uv.cl}

\address{\small $^a$Facultad de Ciencias Básicas, Universidad Tecnológica de Bolívar, Colombia.\\
	$^b$ Institut de Mathématiques de Toulouse, Université Paul Sabatier, France.\\
	$^c$Instituto de Estadística, Universidad de Valparaíso, Chile.}

\begin{abstract}
	We consider  maximum likelihood estimation with data from a bivariate Gaussian process with a separable exponential covariance model under fixed domain asymptotic. We first characterize the equivalence of Gaussian measures under this model. Then consistency and asymptotic distribution for the microergodic parameters are established. A simulation study is presented in order to compare the finite sample behavior of the maximum likelihood estimator with the given asymptotic distribution.
\end{abstract}


\end{frontmatter}


\section{Introduction}

Gaussian processes are widely used in statistics to model spatial data. When fitting a Gaussian field, one has to deal with the issue of the estimation of its covariance. In many cases, a model is chosen for the covariance, which turns the problem into a parametric estimation problem. Within this framework, the maximum likelihood estimator (MLE) of the covariance parameters of a Gaussian stochastic process observed in $\mathbb{R}^d$, $d \geq 1$, has been deeply studied in the last years in the two following asymptotic frameworks. \\
\indent The fixed domain asymptotic framework, sometimes called infill asymptotics \cite{Ste1999, Cre1993}, corresponds to the case where more and more data are observed in some fixed bounded sampling domain (usually a region of $\mathbb{R}^d$). The increasing domain asymptotic framework corresponds to the case where the sampling domain  increases with the number of observed data 
and the distance between any two sampling locations is bounded away from $0$. The asymptotic behavior of the MLE of the covariance parameters can be quite different in these two frameworks \cite{ZhaZim2005}. \vskip .1in

Consider first increasing-domain asymptotics. Then, generally speaking, for all (identifiable) covariance parameters, the MLE is consistent and asymptotically normal under some mild regularity conditions. The asymptotic covariance matrix is equal to the inverse of the (asymptotic) Fisher information matrix. This result was first shown by \cite{MarMar1984}, and then extended in different directions by \cite{cressie93asymptotic,cressie96asymptotics,
	shaby12tapered,Bac2014}. 

The situation is significantly different under fixed domain asymptotics. Indeed, two types of covariance parameters can be distinguished: microergodic and non-microergodic parameters \cite{IbrRoz1978,Ste1999}. A covariance parameter is microergodic if, for two different values of it, the two corresponding Gaussian measures are orthogonal, see \cite{IbrRoz1978,Ste1999}. It is non-microergodic if, even for two different values of it, the two corresponding Gaussian measures are equivalent. Non-microergodic parameters can not be estimated consistently, but misspecifying them asymptotically results in the same statistical inference as specifying them correctly \cite{AEPRFMCF,BELPUICF,UAOLPRFUISOS,ZhaZim2005}. On the other hand, it is at least possible to consistently estimate microergodic covariance parameters, and misspecifying them can have a strong negative impact on inference. 

Nevertheless, under fixed domain asymptotics, it has often proven to be challenging to establish the microergodicity or non-microergodicity of covariance parameters, and to provide asymptotic results for estimators of microergodic parameters. Most available results are specific to particular covariance models. When $d=1$ and the covariance model is exponential, only a reparameterized quantity obtained from the variance and scale parameters is microergodic. It is shown in \cite{Yin1991} that the MLE of this microergodic parameter is consistent and asymptotically normal. When $d>1$ and for a separable exponential covariance function, all the covariance parameters are microergodic, and the asymptotic normality of the MLE is proved in \cite{Yin1993}. Other results in this case are also given in \cite{vdV2010,AbtWel1998,CheSimYin2000}. Consistency of the MLE is shown as well in \cite{LohLam2000} for the scale covariance parameters of the Gaussian covariance function and in \cite{Loh2005} for all the covariance parameters of the separable Mat\'ern $3/2$ covariance function. Finally, for the entire isotropic Mat\'ern class of covariance functions, all parameters are microergodic for $d>4$ \cite{And2010}, and only reparameterized parameters obtained from the scale and variance are microergodic for $d \leq 3$ \cite{Zha2004}. In \cite{ShaKau2013}, the asymptotic distribution of MLEs for these microergodic parameters is provided, generalizing previous results in \cite{DuZhaMan2009} and \cite{WanLoh2011}.  \vskip .1in

All the results discussed above have been obtained when considering a univariate stochastic process. There are few results on maximum likelihood in the multivariate setting. Under increasing-domain asymptotics \cite{BevValVel2015} 
extend the results of \cite{MarMar1984} to the bivariate case and consider the asymptotic distribution of the MLE for a large class of bivariate covariance models in order to test the independence between two Gaussian processes. In \cite{furrer15asymptotic}, asymptotic consistency of the tapered MLE for multivariate processes is established, also under increasing domain asymptotics.
In \cite{PasZha2006}, some results are given on the distribution of the MLE of the correlation parameter between the two components of a bivariate stochastic process with a separable structure, when the space covariance is known, regardless of the asymptotic framework. In \cite{lehrke10large}, the fixed domain asymptotic results of \cite{Yin1993} are extended to the multivariate case, for $d=3$ and when the correlation parameters between the different Gaussian processes are known.
Finally, under fixed domain asymptotics, in the bivariate case and when considering an isotropic  Matérn model, \cite{ZhaCai2015} show which covariance parameters are microergodic. 
\vskip .1in

In this paper, we will extend the results of \cite{Yin1991} (when $d=1$ and the covariance function is exponential) to the bivariate case. First we will consider the equivalence of Gaussian measures, that is to say we will characterize which covariance parameters are microergodic. In the univariate case, \cite{AbtWel1998} characterize the equivalence of Gaussian measures  with exponential covariance function using the entropy distance criteria. We extend their approach to the bivariate case. It turns out, similarly as in the univariate case, that not all covariance parameters are microergodic. Hence not all covariance parameters can be consistently estimated. Then we establish the consistency and the asymptotic normality of the MLE of the microergodic parameters. Some our proof methods are natural extensions of those of \cite{Yin1991} in the univariate case, while others are specific to the bivariate case. \vskip .1in

The paper falls into the following parts. In Section \ref{sec:ind} we characterize the equivalence of Gaussian measures, and describe which covariance parameters are microergodic.
In Section \ref{sec:cons}  we establish the strong consistency of the MLE of the microergodic parameters. Section \ref{sec:norm} is devoted to its asymptotic distribution. Some technical lemmas are needed in order to prove these results and, in particular, Lemma \ref{lem:cross:CLT} is essential to prove the asymptotic normality results. The proofs of the technical lemmas are postponed to the appendix.
Section \ref{sec:simulation} provides a simulation study that shows how well the given 
asymptotic distributions apply to finite sample cases. The final section provides a discussion and open
problems for future research.

%that is 
% $\hat{\rho}  \stackrel{a.s}{\longrightarrow} \rho_{0}$ and
%$\hat{\theta}\hat{\sigma}_{i}^{2} \stackrel{a.s}{\longrightarrow}
%\theta_{0}\sigma_{0i}^{2}$ for $i=1,2$
% of the random vector $\sqrt{n}(\hat{\rho}-\rho_0,\hat{\theta}\hat{\sigma}_{1}^{2}-\theta_{0}\sigma_{01}^{2},\hat{\theta}\hat{\sigma}_{2}^{2}-\theta_{0}\sigma_{02}^{2} )^{\top}$ is given.

%  Define the candidate distance matrix $H_{n}= [h_{ml}]_{m=l=1}^{n}$, where
%$h_{ml}=s_{m}-s_{l}$. The  covariance matrix associated to these observations under the parameters ${\bf \psi}$ 

% ${\bf Z}=(Z_1^{\top}, Z_2^{\top})^{\top} $ with
% $Z_k=(Z_k(s_1),\ldots, Z_k(s_n))^{\top}$ for $k \in \{1,2\}$
%$[C_{\textup{\footnotesize ${ij,\bf Z_{k}}$}}(h,{\bf \psi})]_{i=j=1}^{2}=[\sigma_{\textup{\footnotesize ${i,\bf Z_{k}}$}}\sigma_{\textup{\footnotesize ${j,\bf Z_{k}}$}}(\rho_{\textup{\footnotesize ${\bf Z_{k}}$}}+(1-\rho_{\textup{\footnotesize ${\bf Z_{k}}$}}){\bf 1}_{i=j}) e^{-\theta_{\textup{\footnotesize ${Z_{k}}$}} h}]_{i=j=1}^{2}$. 

\section{Equivalence of Gaussian measures} \label{sec:ind}

First we present some notations used in the whole paper.   If $A=(a_{ij})_{1\leq i\leq k,1\leq j\leq n}$ is a $k\times n$
matrix and $B=(b_{ij})_{1\leq i\leq p,1\leq j\leq q}$ is a $p\times
q$ matrix, then the Kronecker product of the two matrices, denoted by $%
A\otimes B$, is the $kp\times nq$ block matrix%
\begin{equation*}
A\otimes B=%
\begin{bmatrix}
a_{11}B & \dots & a_{1n}B\\
\vdots & \ddots & \vdots\\
a_{k1}B & \dots & a_{kn}B%
\end{bmatrix}%
.
\end{equation*}

In the following, we will consider a stationary zero-mean bivariate Gaussian process observed on fixed compact subset $T$ of $ \mathbb{R}$, ${Z}(s)=\{(Z_{1}(s),Z_{2}(s))^{\top},\, s \in T \}$  with covariance function indexed by a parameter $\psi=(\sigma^{2}_{1},\sigma^{2}_{2},\rho,\theta )^{\top} \in \mathbb{R}^4$, given by  
\begin{eqnarray}
\text{Cov}_{\psi}(Z_{i}(s_l),Z_{j}(s_m))&=&\sigma_{i}\sigma_{j}(\rho +(1-\rho){\bf 1}_{i=j}) e^{-\theta |s_l-s_m|}, \quad i,j=1,2. \label{exp:cov}
\end{eqnarray}
Note that $\sigma_1^{2},\sigma_2^2>0$ are  marginal variances parameters and  $\theta >0$ 
is a  correlation decay parameter.   The quantity $\rho$ with $\vert\rho\vert<1$ is the   so-called colocated correlation parameter \cite{GneKleSch2010},  that expresses the correlation between $Z_1(s)$ and $Z_2(s)$ for each $s$. For  $i=1,2$, the covariance of  the marginal  process $Z_i(s)$ is   $\text{Cov}_{\psi}(Z_{i}(s_l),Z_{i}(s_m))= \sigma_{i}^2e^{-\theta |s_l-s_m|}$.  Such process is  known as the Ornstein-Uhlenbeck process
and it has been widely used to model physical, biological, social, and many other phenomena. 
Denote by $P_{\psi}$ the distribution of the bivariate process $Z$, under covariance parameter $\psi$.
As we consider  fixed domain  asymptotic, the process   ${Z}(s)$ is observed at an increasing number of points on
a compact set $T$. Without loss of generality we consider  $T=[0,1]$ and denote by 
$0 \leq s_{1} < \ldots < s_{n}\leq 1$ the observation points of the process. Let us notice that the points $s_1,\dots,s_n$ are allowed to be permuted when new points are added and that these points are assumed to be dense in $T$ when $n$ tends towards infinity.
The observations can thus be written as  $Z_n=(Z_{1,n}^{\top},Z_{2,n}^{\top})^{\top}$ 
with $Z_{i,n}=(Z_i(s_1),\ldots, Z_i(s_n))^{\top}$ for $i=1,2$.
Hence the observation vector $Z_n$  follows a centered Gaussian distribution  $Z_n \sim N(0,\Sigma (\psi))$  with covariance matrix 
$\Sigma (\psi)=A\otimes R$, given by 
\begin{equation}\label{eq:covmat}
A=\left(
\begin{array}{cc}\sigma_{\textup{\footnotesize ${1}$}}^{2}&\sigma_{\textup{\footnotesize ${1}$}}\sigma_{\textup{\footnotesize ${2}$}}\rho\\
\sigma_{\textup{\footnotesize ${1}$}}\sigma_{\textup{\footnotesize ${2}$}}\rho& \sigma_{\textup{\footnotesize ${2}$}}^{2}
\end{array} \right),\,\,\,  R=\left[e^{-\theta |s_{m}-s_{l}|}\right]_{1\leq m,l\leq n}, \end{equation}
and the
associated likelihood function is given by

\begin{equation}\label{eq:loglik}
f_n(\psi)=(2\pi)^{-n}|\Sigma(\psi)|^{-1/2}e^{-\frac 12{Z_n}\,^\intercal \Sigma (\psi)^{-1}{Z_n}}.
\end{equation}

$ $ \vskip .1in
The aim of this section is to provide a necessary and sufficient condition to warrant equivalence between two Gaussian measures $P_{\psi_1}$ and $P_{\psi_2}$  with
$\psi_i=(\sigma^2_{i,{1}},\sigma^2_{i,{2}},\rho_i,\theta_i )^{\top}$, $i=1,2$.

% In this section we give necessary and sufficient conditions   for the equivalence of the Gaussian measure  $P_{\psi_1}$  with the Gaussian measure $P_{\psi_2}$.
Specifically let us define  the symmetrized entropy
\begin{equation}\label{div}
I_{n}(P_{\psi_1},P_{\psi_2})= E_{\psi_1} \log
\frac{f_n(\psi_1)}{f_n(\psi_2)}+E_{\psi_2} \log
\frac{f_n(\psi_2)}{f_n(\psi_1)}.
\end{equation} 
We assume in this section that the observation points are the terms of a growing sequence in the sense that, at each step, new points are added to the sampling scheme but none is deleted. This assumption ensures that $I_{n}(P_{\psi_1}, P_{\psi_2})$ is an increasing sequence. Hence we may define the limit $I(P_{\psi_1}, P_{\psi_2})= \lim_{n \rightarrow \infty} I_{n}(P_{\psi_1}, P_{\psi_2})$, possibly infinite.
% and $E_{1_{n}} \log\frac{f_{1_{n}}}{f_{2_{n}}}$ is called Kullback divergence of $P_{1}$ from $P_{2}$.
%It is well know that the Gaussian measures are equivalent or orthogonal (see~ \cite{Ste1999}, Theorem 4, page 117). 
Then 
$P_{\psi_1}$ and $P_{\psi_2}$ are either equivalent or orthogonal if and only if $I(P_{\psi_1}, P_{\psi_2}) < \infty$ or $I (P_{\psi_1}, P_{\psi_2}) = \infty$ respectively (see Lemma 3 in page 77 of \cite{IbrRoz1978} whose arguments can be immeditaly extended to the multivariate case).
Using this criterion, the following lemma 
characterizes the equivalence of the Gaussian measures $P_{\psi_1}$ and $P_{\psi_2}$.

\begin{lem}
	%For any bounded infinite set $T \subset \mathbb{R}$, 
	The two measures $P_{\psi_1}$ and $P_{\psi_2}$ are equivalent on the $\sigma$-algebra generated by
	$\{Z(s), \,s \in T\}$, if and only if 
	$\sigma_{i,{1}}^{2}\theta_{1}=\sigma_{i,{2}}^{2}\theta_{2}$, $i=1,2$ and $\rho_{1}=\rho_{2}$ and orthogonal otherwise.
\end{lem}
\begin{proof} Let us introduce $\Delta_{i}=s_i-s_{i-1}$ for $i=2,\dots,n$ and
	note that
	\begin{equation*}
	\sum_{i=2}^n\Delta_{i} \leq 1\quad\text{and}\quad\lim_{n \rightarrow \infty}\max_{2\leq i\leq n} \Delta_{i}=0.
	\end{equation*}
	
	Let  $R_j=\left[e^{-\theta_j |s_{m}-s_{l}|}\right]_{1\leq m, l\leq n}$, $j=1,2$.
	By expanding (\ref{div}) we find that
	\begin{eqnarray*} \label{con}
		I_{n}(P_{\psi_1},P_{\psi_2})&=& \frac{1}{2}
		\left\{\frac{1}{(1-\rho_{2}^{2})}\left[\frac{\sigma_{1,{1}}^{2}}{\sigma_{1,{ 2}}^{2}}-2\frac{\sigma_{1,{1}} \sigma_{2,{1}} \rho_{1} \rho_{2}} {\sigma_{1, {2}}\sigma_{2, {2}}}
		+\frac{\sigma_{2,{1}}^{2}}{\sigma_{2,{2}}^{2}} \right] 	tr(R_{1}
		R_{2}^{-1})\right.\\
		&&\,\,\,\,+\left.
		\frac{1}{(1-\rho_{1}^{2})}\left[\frac{\sigma_{1,{2}}^{2}}{\sigma_{1,{ 1}}^{2}}-2\frac{\sigma_{1,{2}} \sigma_{2,{2}} \rho_{1} \rho_{2}} {\sigma_{1, {1}}\sigma_{2, {1}}}
		+\frac{\sigma_{2,{2}}^{2}}{\sigma_{2,{1}}^{2}} \right] 	tr(R_{2}
		R_{1}^{-1})
		\right\}-2n.
	\end{eqnarray*}
	If  
	$\sigma_{i,{1}}^{2}\theta_{1}=\sigma_{i,2}^{2}\theta_{2}$ and 
	$\rho_{1}=\rho_{2}$ for $i=1,2$ 
	we obtain	
	\begin{eqnarray*} \label{con_bis}
		I_{n}(P_{\psi_1},P_{\psi_2})&=& \frac{\theta_{2}}{\theta_{1}}tr(R_{1}R_{2}^{-1})+\frac{\theta_{1}}{\theta_{2}}tr(R_{2}R_{1}^{-1})-2n.		
	\end{eqnarray*}
	In order to compute $tr(R_{1}R_{2}^{-1})$ and $tr(R_{2}R_{1}^{-1})$, we use some results in  \cite{AntZag2010}.
	The matrix $R_j$ can be written as follows,
	$$R_{j}=\left(\begin{matrix} 1& e^{-\theta_{j}\Delta_2} & \cdots & e^{-\theta_{j}\sum\limits_{i=2}^n\Delta_{i}} \\
	e^{-\theta_{j}\Delta_2}& 1 &  \cdots& e^{-\theta_{j}\sum\limits_{i=3}^n\Delta_{i}}\\
	\vdots& \vdots & \ddots& \vdots\\
	e^{-\theta_{j}\sum\limits_{i=2}^n\Delta_{i}}&  
	e^{-\theta_{j}\sum\limits_{i=3}^n\Delta_{i}} & \cdots & 1\\
	\end{matrix}\right)$$
	and
	$R_j^{-1}$ can be written as
	$$R_{j}^{-1}=\left(\begin{matrix} \frac{1}{1-e^{-2\theta_{j}\Delta_2}}&\frac{-e^{-\theta_{j}\Delta_2}}{1-e^{-2\theta_{j}\Delta_2}}& 0&\cdots&0\\
	\frac{-e^{-\theta_{j}\Delta_2}}{1-e^{-2\theta_{j}\Delta_2}}& \frac{1}{1-e^{-2\theta_{j}\Delta_2}}+\frac{e^{-2\theta_{j}\Delta_3}}{1-e^{-2\theta_{j}\Delta_3}}&\ddots& \ddots &\vdots\\
	0 & \ddots & \ddots &  & 0\\
	\vdots & \ddots &  & \frac{1}{1-e^{-2\theta_{j}\Delta_{n-1}}}+\frac{e^{-2\theta_{j}\Delta_n}}{1-e^{-2\theta_{j}\Delta_n}}& \frac{-e^{-\theta_{j}\Delta_n}}{1-e^{-2\theta_{j}\Delta_n}}\\
	0& \cdots & 0 & \frac{-e^{-\theta_{j}\Delta_n}}{1-e^{-2\theta_{j}\Delta_n}}& 
	\frac{1}{1-e^{-2\theta_{j}\Delta_n}}
	\end{matrix} \right).$$
	\\
	Since,
	$tr(R_{j}R_{k}^{-1})=\sum\limits_{i=1}^{n}\sum\limits_{m=1}^{n} (R_{j} \otimes R_{k}^{-1})_{im},\,\,j,k=1,2, j \neq k$, we have
	\begin{eqnarray*}
		tr(R_{j}
		R_{k}^{-1})&=&-2\sum_{i=2}^n\frac{e^{-(\theta_{j}+\theta_{k})\Delta_{i}}}{1-e^{-2\theta_{k}\Delta_{i}}}   + \sum_{i=2}^n\frac{1}{1-e^{-2\theta_{k}\Delta_{i}}}+ \sum_{i=3}^n\frac{e^{-2\theta_{k}\Delta_{i}}}{1-e^{-2\theta_{k}\Delta_{i}}}+\frac{1}{1-e^{-2\theta_{k}\Delta_2}}\\
		&=&
		\sum_{i=2}^n\frac{-2e^{-(\theta_{j}+\theta_{k})\Delta_{i}}+1+e^{-2\theta_{k}\Delta_{i}} }{1-e^{-2\theta_{k}\Delta_{i}}}+1\\
		&=&\sum_{i=2}^n \frac{(e^{-\theta_{k}\Delta_{i}}-e^{-\theta_{j}\Delta_{i}})^{2}}{1-e^{-2\theta_{k}\Delta_{i}}}+\sum_{i=2}^n\frac{1-e^{-2\theta_{j}\Delta_{i}}}{1-e^{-2\theta_{k}\Delta_{i}}}+1.
	\end{eqnarray*}		
	
	%Then note that in the case of the exponential correlation function, we have (see~ \cite{AbtWel1998}),
	\noindent Then, we can write $I_{n}(P_{\psi_1},P_{\psi_2})$ as
	\begin{eqnarray*}  \label{con1}
		\label{con2_bis} I_{n}(P_{\psi_1},P_{\psi_2})	&=&\frac{\theta_{2}}{\theta_{1}}\left(\sum_{i=2}^n \frac{(e^{-\theta_{1}\Delta_{i}}-e^{-\theta_{2}\Delta_{i}})^{2}}{1-e^{-2\theta_{2}\Delta_{i}}}+\sum_{i=2}^n\frac{1-e^{-2\theta_{1}\Delta_{i}}}{1-e^{-2\theta_{2}\Delta_{i}}}+1\right)\\
		&&+\frac{\theta_{1}}{\theta_{2}}\left(\sum_{i=2}^n \frac{(e^{-\theta_{2}\Delta_{i}}-e^{-\theta_{1}\Delta_{i}})^{2}}{1-e^{-2\theta_{1}\Delta_{i}}}+\sum_{i=2}^n\frac{1-e^{-2\theta_{2}\Delta_{i}}}{1-e^{-2\theta_{1}\Delta_{i}}}+1\right)-2n.\\
	\end{eqnarray*}	
	For $j,k=1,2$, $j\neq k$, as is obtained by Taylor expansion, since $\max_i\Delta_i$ tends to $0$, we have
	\begin{equation*}
	\max_{2\leq i\leq n}\left\vert\frac{1-e^{-2\theta_{j}\Delta_{i}}}{\Delta_i(1-e^{-2\theta_{k}\Delta_{i}})}-\frac{\theta_j}{\Delta_i\theta_k}\right\vert=O(1)
	\quad\text{and}\quad
	\max_{2\leq i\leq n}\frac{(e^{-\theta_{j}\Delta_{i}}-e^{-\theta_{k}\Delta_{i}})^{2}}{\Delta_i(1-e^{-2\theta_{k}\Delta_{i}})}=O(1)
	\end{equation*}
	Since $\sum_{i}\Delta_i$ tends to $1$,
	\begin{equation*}
	I_{n}(P_{\psi_1},P_{\psi_2})=\frac{\theta_2}{\theta_1}(1+O(1))+\frac{\theta_1}{\theta_2}(1+O(1))
	\end{equation*}
	and
	$I(P_{\psi_1}, P_{\psi_2}) = \lim_{n \rightarrow \infty} I_{n}(P_{\psi_1},P_{\psi_2}) < \infty$.\\
	Then the two Gaussian measures $P_{\psi_1}$ and $P_{\psi_2}$  are equivalent on the $\sigma-$algebra generated by $Z$ if and only if $\sigma_{i,1}^{2}\theta_{1}=\sigma_{i,2}^{2}\theta_{2}$, $i=1,2$, and 	$\rho_{1}=\rho_{2}$. .
	%as implied by Theorem \ref{thm:consistency}.
\end{proof}
Note that   sufficient conditions for the equivalence of Gaussian measures using a  generalization of the covariance model (\ref{exp:cov}) are given in \cite{ZhaCai2015}. A consequence of the previous lemma is that it is not possible to  estimate consistently all the parameters
individually if the data are observed on a compact set $T$.  However the microergodic parameters
$\sigma_1^2\theta$, $\sigma_2^2\theta$ and $\rho$ are consistently estimable. The following section is devoted to their estimation.

\section{Consistency of the Maximum Likelihood Estimator}\label{sec:cons}
Let $\widehat{\psi}=(\hat{\theta},\hat{\sigma}_{1}^{2},\hat{\sigma}_{2}^{2},\hat{\rho})^{\top}$ be the MLE obtained by maximizing  $f_n(\psi)$ with respect to $\psi$. In the rest of the paper, we will denote by 
$\theta_0$, $\sigma^2_{i0}$, $i=1,2$ and $\rho_0$ the true but unknown parameters that have to be estimated. We let $var = var_{\psi_0}$, $cov = cov_{\psi_0}$ and $\mathbb{E} = \mathbb{E}_{\psi_0}$ denote the variance, covariance and expectation under $P_{\psi_0}$.
In this section, we establish the strong consistency of   $\hat{\rho}$, $\hat{\theta}\hat{\sigma}_{1}^{2}$   and $\hat{\theta}\hat{\sigma}_{2}^{2}$ , that is the MLE of the microergodic parameters.

We first consider an explicit expression for the negative  log-likelihood   function
\begin{equation}\label{log}
l_{n}(\psi)=- 2\log(f_n(\psi))=   2n \log (2\pi)+\log|\Sigma(\psi)|+ Z_n^{\top} \left[ \Sigma
(\psi)\right]^{-1} Z_n.
\end{equation}
The  explicit expression is given in the following lemma whose proof can be found in the appendix.
\begin{lem} \label{lem:expression:likelihood}
	The negative log-likelihood function in Equation  (\ref{log}) can be written as
	
	\begin{eqnarray*}
		l_{n}(\psi)&=&n \left[\log (2\pi)+\log(1-\rho^{2})\right]+\sum_{k=1}^{2} \log
		(\sigma_{k}^{2}) + \sum_{k=1}^{2} \sum_{i=2}^{n} \log
		\left[\sigma_{k}^{2}\left(1-e^{-2\theta \Delta_{i} }\right)\right]
		\\
		&&+\frac{1}{1-\rho^{2}}\left\{ \sum_{k=1}^{2}  \frac{1}{\sigma_{k}^{2}}
		\left(z_{k,1}^{2}+\sum_{i=2}^{n} \frac{\left(z_{k,i}-e^{-\theta
				\Delta_{i}}z_{k,i-1}\right)^{2}}{1-e^{-2\theta \Delta_{i}}}
		\right)\right.\\
		&&\,\,\,\,\,\,\,\,\,\,\,\,\,\,\,\,\,\,\,\,\,\,\,\,\,\,\,\,\,-\left.
		\frac{2\rho}{\sigma_{1}\sigma_{2}}\left(z_{1,1}z_{2,1}+\sum_{i=2}^{n} \frac{\left(z_{1,i}-e^{-\theta
				\Delta_{i}}z_{1,i-1}\right)\left(z_{2,i}-e^{-\theta
				\Delta_{i}}z_{2,i-1}\right)}{1-e^{-2\theta \Delta_{i}}}
		\right) \right\},
	\end{eqnarray*}
	with $z_{k,i}=Z_k(s_i)$  and  $\Delta_{i}=s_{i}-s_{i-1},\,\,i=2,\ldots,n$.\\
\end{lem}

The following theorem uses Lemma \ref{lem:expression:likelihood} in order to establish strong consistency of MLE of 
the microergodic parameters $\rho$, $\theta \sigma_{1}^{2}$, $\theta \sigma_{2}^{2}$.

\begin{thm} \label{thm:consistency}
	Let $J = (a_{\theta},
	b_{\theta})\times (a_{\sigma_{1}},b_{\sigma_{1}}) \times (a_{\sigma_{2}},b_{\sigma_{2}})  \times (a_{\rho},b_{\rho})$, with $0 < a_{\theta}\leq \theta_0 \leq b_{\theta}<\infty,\,\,0 < a_{\sigma_{1}}\leq \sigma_{01}^2 \leq b_{\sigma_{1}}<\infty, \,\,0 < a_{\sigma_{2}}\leq  \sigma_{02}^2 \leq b_{\sigma_{2}}<\infty $ and $-1 < a_{\rho}\leq \rho_0 \leq  b_{\rho}<1$.
	Define
	$\widehat{\psi}=(\hat{\theta},\hat{\sigma}_{1}^{2},\hat{\sigma}_{2}^{2},\hat{\rho})$
	as the minimum of the negative  log-likelihood estimator, solution of
	\begin{equation}\label{eq1}
	l_{n}(\widehat{\psi})=
	\underset{{\psi} \in
		J}{\min}\,\,l_{n} (\psi).
	\end{equation}
	Then, with probability one,
	$\widehat{\psi}$
	exist for $n$ large enough and when $n \rightarrow + \infty$
	\begin{eqnarray}\label{eq2}
	\hat{\rho} &\stackrel{a.s}{\longrightarrow}&\rho_{0},\,\,\,\,\,\\\label{eq3}
	\hat{\theta}\hat{\sigma}_{1}^{2} &\stackrel{a.s}{\longrightarrow}&
	\theta_{0}\sigma_{01}^{2},\,\,\,\,\,\\\label{eq4}
	\hat{\theta}\hat{\sigma}_{2}^{2} &\stackrel{a.s}{\longrightarrow}&
	\theta_{0}\sigma_{02}^{2}.\,\,\,\,\,
	\end{eqnarray}
\end{thm}
\begin{proof}
	The proof follows the guideline of the consistency of the maximum likelihood estimation given in \cite{Yin1991}.
	Hence consistency results given in (\ref{eq2}), (\ref{eq3}) and (\ref{eq4}) hold as long as we can prove that
	there exist $0<d<D< \infty$ such that for every $\epsilon > 0$,  $\psi$ and $\widetilde{\psi}$, with $\|
	\psi-\widetilde{\psi}\|>\epsilon$
	%\begin{equation}\label{result}
	%\underset{{0<\theta \leq d,\,\, a_{\sigma_{1}} \leq \sigma_{1}^{2} \leq b_{\sigma_{1}} ,\,\, a_{\sigma_{2}} 
	%\leq \sigma_{2}^{2}\leq b_{\sigma_{2}} },\,\, -1 < \rho <
	%1}{\inf}\left\{l_{n}(\psi)-
	%l_{n}(\widetilde{\psi})
	%\right\}\rightarrow \infty \,\,\, a.s.
	%\end{equation}
	\begin{equation}\label{result}
	\underset{\{\psi \in J ,\: \|
		\psi-\widetilde{\psi}\|>\epsilon \}}{\min}\left\{l_{n}(\psi)-
	l_{n}(\widetilde{\psi})
	\right\}\rightarrow \infty \,\,\, a.s.
	\end{equation}
	
	where
	$\widetilde{\psi}=(\widetilde{\theta},\widetilde{\rho}^{2},\widetilde{\sigma}_{1}^{2},\widetilde{\sigma}_{2}^{2})^{\top} \in
	J$ can be any nonrandom vector such that
	$$
	\widetilde{\rho}=\rho_{0},\quad 
	\widetilde{\theta}\widetilde{\sigma}_{1}^{2}=
	\theta_{0}\sigma_{01}^{2},\quad
	\widetilde{\theta}\widetilde{\sigma}_{2}^{2}=
	\theta_{0}\sigma_{02}^{2}.$$
	In order to simplify our notation, let 
	$W_{k,i,n}=\frac{z_{k,i}-e^{-\theta_{0}\Delta_{i}}z_{k,i-1}
	}{\left[\sigma^{2}_{0k} (1-e^{-2\theta_{0}\Delta_{i}})
	\right]^{\frac{1}{2}}}$, $k=1,2,\,\,i=2,...,n$.
By the Markovian and Gaussian properties of $Z_{1}$ and $Z_{2}$, it
follows that for each $i \geq 2$, $W_{k,i,n}$ is independent of
$\left\{Z_{k,j}, j\leq i-1  \right\},\,\,k=1,2$. Moreover $\left\{W_{k,i,n},
2\leq i \leq n \right\},\,\, k=1,2$  are an i.i.d. sequences of standard Gaussian random variables.
Using Lemma \ref{lem:expression:likelihood} we write,
\begin{eqnarray*}
	l_{n}(\psi)&=&\sum_{k=1}^{2}  \sum_{i=2}^{n} \log \left[\sigma_{k}^{2}\left(1-e^{-2\theta
		\Delta_{i}}\right)\right]+\frac{1}{1-\rho^{2}}\sum_{k=1}^{2} \sum_{i=2}^{n}
	\frac{\left(z_{k,i}-e^{-\theta
			\Delta_{i}}z_{k,i-1}\right)^{2}}{\sigma_{k}^{2}(1-e^{-2\theta\Delta_{i}})}\\&&
	+\frac{2\rho}{(1-\rho^{2})}\sum_{i=2}^{n} \frac{\left(z_{1,i}-e^{-\theta
			\Delta_{i}}z_{1,i-1}\right)\left(z_{2,i}-e^{-\theta
			\Delta_{i}}z_{2,i-1}\right)}{\sigma_{1}\sigma_{2}(1-e^{-2\theta \Delta_{i}})}+n \log (2\pi)+c(\psi,n),
\end{eqnarray*}

with $c(\psi,n)=\sum_{k=1}^{2}  \log \sigma_{k}^{2}+ n \log \left(1-\rho^{2}\right) +\frac{1}{1-\rho^2}\left[\sum_{k=1}^{2}  \frac{z^2_{k,1}}{\sigma^2_{k}} - 2\rho \frac{z_{1,1}z_{2,1}}{\sigma_{1}\sigma_{2}}\right]$
and from the proof of Theorem 1 in~\cite{Yin1991}, uniformly in $0<\theta\leq d_{\theta}$ and  $\sigma_{k}^{2} \in
[a_{\sigma_{k}} ,b_{\sigma_{k}} ]$, $k=1,2:$\\ %and $a \leq \sigma_{2}^{2} \leq b$
\begin{equation}\label{a}
\sum_{i=2}^{n} \frac{\left(z_{k,i}-e^{-\theta
		\Delta_{i}}z_{k,i-1}\right)^{2}}{\sigma_{k}^{2}(1-e^{-2\theta\Delta_{i}})}=\sum_{i=2}^{n}
\frac{\sigma_{0k}^2(1-e^{-2\theta_{0}\Delta_{i}})}{\sigma_{k}^2(1-e^{-2\theta\Delta_{i}})}W_{k,i,n}^{2}+O(n^{\frac{1}{2}}),\,\, k=1,2.
\end{equation}\\
%$$ \sum_{i=2}^{n} \frac{\left(z_{2,i}-e^{-\theta
%\Delta_{i}}z_{2,i-1}\right)^{2}}{\sigma_{2}^{2}(1-e^{-2\theta\Delta_{i}})}=\sum_{i=2}^{n}
%\frac{\sigma_{02}^2(1-e^{-2\theta_{0}\Delta_{i}})}{\sigma_{2}^2(1-e^{-2\theta\Delta_{i}})}V_{i,n}^{2}+O(n^{\frac{1}{2}})
%$$
%Then, we have
%\begin{eqnarray*}
%l_{n}(\theta,\rho^{2}_{12},\sigma_{1}^{2},\sigma_{2}^{2})&=&n \log (2\pi)+c+\sum_{i=2}^{n} \log \left[\sigma_{1}^{2}\left(1-e^{-2\theta
%\Delta_{i}}\right)\right]+\sum_{i=2}^{n} \log
%\left[\sigma_{2}^{2}\left(1-e^{-2\theta\Delta_{i}}\right)\right]\\
%&&+\frac{1}{1-\rho_{12}^{2}}\sum_{i=2}^{n}
%\frac{\sigma_{01}^2(1-e^{-2\theta_{0}\Delta_{i}})}{\sigma_{1}^2(1-e^{-2\theta\Delta_{i}})}W_{i,n}^{2}+\frac{1}{1-\rho_{12}^{2}}
%O(n^{\frac{1}{2}})
%\\
%&&
%+\frac{1}{1-\rho_{12}^{2}}\sum_{i=2}^{n}
%\frac{\sigma_{02}^2(1-e^{-2\theta_{0}\Delta_{i}})}{\sigma_{2}^2(1-e^{-2\theta\Delta_{i}})}V_{i,n}^{2}+
%\frac{1}{1-\rho_{12}^{2}} O(n^{\frac{1}{2}})
%\\
%&&+\frac{2\rho_{12}}{1-\rho_{12}^{2}}\sum_{i=2}^{n}\frac{\left(z_{1,i}-e^{-\theta
%\Delta_{i}}z_{1,i-1}\right)\left(z_{2,i}-e^{-\theta
%\Delta_{i}}z_{2,i-1}\right)}{\sigma_{1}\sigma_{2}(1-e^{-2\theta\Delta_{i}})}
%\end{eqnarray*}\\
Moreover, from Cauchy-Schwarz inequality,
\begin{eqnarray}\label{b}
\sum_{i=2}^{n} \frac{\left(z_{1,i}-e^{-\theta
		\Delta_{i}}z_{1,i-1}\right)\left(z_{2,i}-e^{-\theta
		\Delta_{i}}z_{2,i-1}\right)}{\sigma_{1}\sigma_{2}(1-e^{-2\theta \Delta_{i}})}
&\leq&\prod_{k=1}^2\left(\sum_{i=2}^{n}
\frac{\sigma_{0k}^2(1-e^{-2\theta_{0}\Delta_{i}})}{\sigma_{k}^2(1-e^{-2\theta\Delta_{i}})}W_{k,i,n}^{2}+O(n^{\frac{1}{2}})\right)^\frac{1}{2}\nonumber,\\
\end{eqnarray}\\
and from Lemma 2(ii) in \cite{Yin1991} uniformly in $\theta \leq R$ and $\sigma_{k}^{2} \in
[a_{\sigma_{k}} ,b_{\sigma_{k}} ]$, for every $\alpha_k >0$, with $k=1,2$,\\
\begin{equation}\label{c}
\sum_{i=2}^{n}
\frac{\sigma_{0k}^2(1-e^{-2\theta_{0}\Delta_{i}})}{\sigma_{k}^2(1-e^{-2\theta\Delta_{i}})}W_{k,i,n}^{2}
=
\frac{\sigma_{0k}^{2}\theta_{0}}{\sigma_{k}^{2}\theta}(n-1)+\frac{1}{\theta}
O(n^{\frac{1}{2}+\alpha_{k}}),\,\,k=1,2.
\end{equation}\\
%Then $\sum_{i=2}^{n}
%\frac{\sigma_{01}^2(1-e^{-2\theta_{0}\Delta_{i}})}{\sigma_{1}^2(1-e^{-2\theta\Delta_{i}})}W_{i,n}^{2}+O(n^{\frac{1}{2}})=\frac{\sigma_{01}^{2}\theta_{0}}{\sigma_{1}^{2}\theta}(n-1)+
%\frac{1}{\theta}O(n^{\frac{1}{2}+\alpha_{1}})$, in the same form
%$\sum_{i=2}^{n}
%\frac{\sigma_{02}^2(1-e^{-2\theta_{0}\Delta_{i}})}{\sigma_{2}^2(1-e^{-2\theta\Delta_{i}})}V_{i,n}^{2}+O(n^%{\frac{1}{2}})=\frac{\sigma_{02}^{2}\theta_{0}}{\sigma_{2}^{2}\theta}(n-1)+
%\frac{1}{\theta}O(n^{\frac{1}{2}+\alpha_{2}})$. 
Combining (\ref{a}), (\ref{b}) and (\ref{c}), we can write,
\begin{eqnarray*}
	l_{n}(\psi)&\geq&n \log (2\pi)+c(\psi,n)+\sum_{k=1}^{2}  \sum_{i=2}^{n} \log \left[\sigma_{k}^{2}\left(1-e^{-2\theta
		\Delta_{i}}\right)\right]+\frac{1}{1-\rho^{2}}
	O(n^{\frac{1}{2}})\\
	&&+\frac{1}{1-\rho^{2}}\sum_{k=1}^{2} \sum_{i=2}^{n}
	\frac{\sigma_{0k}^2(1-e^{-2\theta_{0}\Delta_{i}})}{\sigma_{k}^2(1-e^{-2\theta\Delta_{i}})}W_{k,i,n}^{2}
	\\
	&&-\frac{2\rho}{1-\rho^{2}}n\prod_{k=1}^2 
	\left(\frac{\sigma_{0k}^{2}\theta_{0}}{\sigma_{k}^{2}\theta}+
	\frac{1}{n\theta}O(n^{\frac{1}{2}+\alpha_{k}})-\frac{1}{n}\frac{\sigma_{0k}^{2}\theta_{0}}{\sigma_{k}^{2}\theta}\right)^{\frac{1}{2}}.
\end{eqnarray*}
Therefore
\begin{eqnarray*}
	l_{n}(\psi)-l_{n}(\widetilde{\psi})&\geq&
	p(\psi,\widetilde{\psi},n)+\sum_{k=1}^{2} \sum_{i=2}^{n} log \left[ \frac{\sigma_{k}^{2}(1-e^{-2\theta
			\Delta_{i}})}{\widetilde{\sigma}_{k}^{2}(1-e^{-2\widetilde{\theta}\Delta_{i}})}\right]\\
	&+& \sum_{k=1}^{2} \sum_{i=2}^{n}
	\left[\frac{1}{1-\rho^{2}} \frac{\sigma_{0k}^{2}(1-e^{-2\theta_{0}
			\Delta_{i}})}{\sigma_{k}^{2}(1-e^{-2\theta\Delta_{i}})}
	-\frac{1}{1-\widetilde{\rho}^{2}}\frac{\sigma_{0k}^{2}(1-e^{-2\theta_{0}\Delta_{i}})}{\widetilde{\sigma}_{k}^{2}(1-e^{-2\widetilde{\theta}\Delta_{i}})}
	\right] W_{k,i,n}^{2}\\
	&-&\frac{2\rho}{1-\rho^{2}}\frac{n}{\theta} \prod_{k=1}^2
	\left(\frac{\sigma_{0k}^{2}\theta_{0}}{\sigma_{k}^{2}}+
	\frac{1}{n}O(n^{\frac{1}{2}+\alpha_{k}})-\frac{1}{n}\frac{\sigma_{0k}^{2}\theta_{0}}{\sigma_{k}^{2}}\right)^{\frac{1}{2}}\\
	%&&\,\,\,\,\,\,\,\,\,\,\,\,\,\,\,\,\,\,\,\,\,\,\,\,\times
	%\left(\frac{\sigma_{02}^{2}\theta_{0}}{\sigma_{2}^{2}}+
	%\frac{1}{n}O(n^{\frac{1}{2}+\alpha_{2}})-\frac{1}{n}\frac{\sigma_{02}^{2}\theta_{0}}{\sigma_{2}^{2}}
	%\right)^{\frac{1}{2}}\\
	&+&\frac{2\widetilde{\rho}}{1-\widetilde{\rho}^{2}}n \prod_{k=1}^2 \left(1+
	\frac{1}{n \widetilde{\theta}}O(n^{\frac{1}{2}+\alpha_{k}})-\frac{1}{n}\right)^{\frac{1}{2}}\\
	%\times
	%\left(1+\frac{1}{n \widetilde{\theta}} O(n^{\frac{1}{2}+\alpha_{2}})-\frac{1}{n}
	%\right)^{\frac{1}{2}}\\
	%&+&
	&+&\frac{1}{1-\rho^{2}}O(n^{\frac{1}{2}})+\frac{1}{1-\widetilde{\rho}^{2}}O(n^{\frac{1}{2}}),
\end{eqnarray*}
where $p(\psi,\widetilde{\psi},n)=c(\psi,n)-c(\widetilde{\psi},n)$. From lemma 2 in \cite{Yin1991}, for some $M_{k}>0$ and uniformly in $\theta \leq R$ and $\sigma_{k}^{2} \in
[a_{\sigma_{k}} ,b_{\sigma_{k}} ]$, $k=1,2$
\begin{equation}\label{d}
\sum_{i=2}^{n} \log \left[ \frac{\sigma_{k}^{2}(1-e^{-2\theta \Delta_{i}})}{\widetilde{\sigma}_{k}^{2}(1-e^{-2\widetilde{\theta}\Delta_{i}})}\right]\geq
\sum_{i=2}^{n} \log \left(\frac{\theta}{M_{k}}\right)= (n-1) \log
\left(\frac{\theta}{M_{k}}\right),\,\,k=1,2.
\end{equation}
and
\begin{eqnarray} \label{f} 
\sum_{i=2}^{n}
\left[\frac{1}{1-\rho^{2}} \frac{\sigma_{0k}^{2}(1-e^{-2\theta_{0}
		\Delta_{i}})}{\sigma_{k}^{2}(1-e^{-2\theta\Delta_{i}})}\right.
&-&\left.\frac{1}{1-\widetilde{\rho}^{2}}\frac{\sigma_{0k}^{2}(1-e^{-2\theta_{0}\Delta_{i}})}{\widetilde{\sigma}_{k}^{2}(1-e^{-2\widetilde{\theta}\Delta_{i}})}
\right] W_{k,i,n}^{2}\\&=&
(n-1)
\left( \frac{1}{1-\rho^{2}}    \frac{\sigma_{0k}^{2}\theta_{0}}{\sigma_{k}^{2}\theta}-\frac{1}{1-\widetilde{\rho}^{2}}\right)+\theta^{-1}O(n^{\frac{1}{2}+\alpha_{k}}),\,\, k=1,2. \nonumber
\end{eqnarray}
%$$\sum_{i=2}^{n}
%\left[ \frac{1}{1-\rho_{12}^{2}}\frac{\sigma_{02}^{2}(1-e^{-2\theta_{0}
%\Delta_{i}})}{\sigma_{2}^{2}(1-e^{-2\theta\Delta_{i}})}
%-\frac{1}{1-\widetilde{\rho}^{2}_{12}}\frac{\sigma_{02}^{2}(1-e^{-2\theta_{0}\Delta_{i}})}{\widetilde{\sigma}_{2}^{2}(1-e^{-2\widetilde{\theta}\Delta_{i}})}
%\right] V_{i,n}^{2}=(n-1)
%\left(\frac{1}{1-\rho_{12}^{2}}\frac{\sigma_{02}^{2}\theta_{0}}{\sigma_{2}^{2}\theta}-\frac{1}{1-\widetilde{\rho}^{2}_{12}}\right)+\theta^{-1}O(n^{\frac{1}{2}+\alpha})$$
Let $ \widetilde{\rho} = \min \left\{\widetilde{\rho}, \rho \right\}$ and combining (\ref{d}) and (\ref{f}), we can write, 
\begin{eqnarray*}
	l_{n}(\theta,\rho,\sigma_{1}^{2},\sigma_{2}^{2})-l_{n}(\widetilde{\theta},\widetilde{\rho},\widetilde{\sigma}_{1}^{2},\widetilde{\sigma}_{2}^{2})&\geq&
	p(\psi,\widetilde{\psi},n) +  \sum_{k=1}^{2}\frac{n}{\theta
		(1-\rho_{0}^{2})}\left[\frac{\sigma_{0k}^{2}\theta_{0}}{\sigma_{k}^{2}}+O(n^{\gamma_{k}-1})
	\right]\\&&-\sum_{k=1}^{2}\frac{1}{1-\rho_{0}^{2}}\frac{\sigma_{0k}^{2}\theta_{0}}{\sigma_{k}^{2}\theta}+\sum_{k=1}^{2}(n-1)\left[log
	\left(\frac{M_{k}}{\theta}\right)-\frac{1}{1-\rho_{0}^{2}}
	\right]\\
	&&-\frac{2\rho_{0}}{1-\rho_{0}^{2}}\frac{n}{\theta}
	\prod_{k=1}^2\left(\frac{\sigma_{0k}^{2}\theta_{0}}{\sigma_{k}^{2}}+
	\frac{1}{n}O(n^{\frac{1}{2}+\alpha_{k}})-\frac{1}{n}\frac{\sigma_{0k}^{2}\theta_{0}}{\sigma_{k}^{2}}\right)^{\frac{1}{2}}\\
	%\,\,\,\,\,\,\,\,\,\,\,\,\,\,\,\,\,\,\,\,\,\,\,\,\,\,\times\left(\frac{\sigma_{02}^{2}\theta_{0}}{\sigma_{2}^{2}}+
	%\frac{1}{n}O(n^{\frac{1}{2}+\alpha_{2}})
	%-\frac{1}{n}\frac{\sigma_{02}^{2}\theta_{0}}{\sigma_{2}^{2}}
	%\right)^{\frac{1}{2}}\\
	&&+\frac{2\rho_{0}}{1-\rho^{2}_{0}}n \prod_{k=1}^2\left(1+
	\frac{1}{n \widetilde{\theta}}O(n^{\frac{1}{2}+\alpha_{k}})-\frac{1}{n}\right)^{\frac{1}{2}}.
	%&&\,\,\,\,\,\,\,\,\,\,\,\,\,\,\,\,\,\,\,\,\,\,\,\,\,\,\times
	%\left(1+\frac{1}{n \widetilde{\theta}} O(n^{\frac{1}{2}+\alpha_{2}})-\frac{1}{n}
	%\right)^{\frac{1}{2}}.
\end{eqnarray*}
for some $\gamma_{k}<1,\,\,k=1,2$,  where the $O(n^{\gamma_k})$
term is uniform in $\theta \leq R$.\\
Since some $log(\theta)^{-1}=o(\theta^{-1})$ as $\theta \downarrow
0$, we can choose $\delta_{\theta}$ small enough so that for all $\theta
\leq \delta_{\theta}$,
$\frac{\sigma_{0}^{2}\theta_{0}}{2b_{\sigma}\theta}-1-log(\frac{M}{\theta})\geq
\eta$, which implies that with probability 1, \\
\begin{eqnarray*}
	l_{n}(\theta,\rho,\sigma_{1}^{2},\sigma_{2}^{2})-l_{n}(\widetilde{\theta},\widetilde{\rho},\widetilde{\sigma}_{1}^{2},\widetilde{\sigma}_{2}^{2})&\geq&
	p(\psi,\widetilde{\psi},n) +\frac{n}{
		(1-\rho_{0}^{2})}\left\{\frac{1}{\theta}\sum_{k=1}^{2}\frac{\sigma_{0k}^{2}\theta_{0}}{2b_{\sigma_{k}}}+O(n^{\gamma_{k}-1})
	\right.\\
	&&\left.- 2\rho_{0} \left[ \prod_{k=1}^2\frac{1}{\theta}
	\left(\frac{\sigma_{0k}^{2}\theta_{0}}{\sigma_{k}^{2}}+
	O(n^{\alpha_{k}-\frac{1}{2}})-\frac{1}{n}\frac{\sigma_{0k}^{2}\theta_{0}}{\sigma_{k}^{2}}\right)^{\frac{1}{2}}\right. \right.\\
	&&\,\,\,\,\,\,\,\,\,\,\,\,\,\,\,\,\,\,\,\,\,\,\,\,\,\,\,\,\left.- \prod_{k=1}^2\left(1+
	\left.\frac{1}{\widetilde{\theta}}O(n^{\alpha_{k}-\frac{1}{2}})-\frac{1}{n}\right)^{\frac{1}{2}}\right]\right\}.
\end{eqnarray*}
Thus we get (\ref{result}) by letting $d_{\theta}=\delta_{\theta}$.\\
%%FIXME : voir si on ne peut pas remplacer d par \delta partout %%%%
Hence, since $p(\psi,\widetilde{\psi},n) \rightarrow  0$,  

\begin{equation*}
\frac{1}{\theta}\sum_{k=1}^{2}\frac{\sigma_{0k}^{2}\theta_{0}}{2b_{\sigma_{k}}}+O(n^{\gamma_{k}-1})\leq  \infty
\end{equation*}
and
\begin{equation*}
- 2\rho_{0} \left[ \prod_{k=1}^2\frac{1}{\theta}
\left(\frac{\sigma_{0k}^{2}\theta_{0}}{\sigma_{k}^{2}}+
O(n^{\alpha_{k}-\frac{1}{2}})-\frac{1}{n}\frac{\sigma_{0k}^{2}\theta_{0}}{\sigma_{k}^{2}}\right)^{\frac{1}{2}}- \prod_{k=1}^2\left(1+
\frac{1}{\widetilde{\theta}}O(n^{\alpha_{k}-\frac{1}{2}})-\frac{1}{n}\right)^{\frac{1}{2}}\right] \leq  \infty,
\end{equation*}
we prove that
\begin{equation*}
\underset{\{\psi \in J ,\: \|
	\psi-\widetilde{\psi}\|>\epsilon \}}{\min}\left\{l_{n}(\theta,\rho,\sigma_{1}^{2},\sigma_{2}^{2})-
l_{n}(\widetilde{\theta},\widetilde{\rho},\widetilde{\sigma}_{1}^{2},\widetilde{\sigma}_{2}^{2})
\right\}\rightarrow \infty
\end{equation*}
when $n \rightarrow \infty$, uniformly in $\theta \leq \delta_{\theta}$.
\end{proof}

\section{Asymptotic distribution}\label{sec:norm}
Before we state the main result on the MLE  asymptotic distribution, we need to introduce some notation
that will be used throughout this paper.
Because of Theorem \ref{thm:consistency}, there exists a compact subset $\mathcal{S}$  of $(0,+\infty)\times(0,+\infty)\times (0,+\infty)\times(-1,1)$ of the form $\Theta  \times \mathcal{V} \times \mathcal{V}\times \mathcal{R}$, such that a.s. $\hat{\psi}$ belongs to   $\mathcal{S}$ for $n$ large enough. We let $O_u(1)$ denote any real function $g_n(\theta,\rho,\sigma_1^2,\sigma_2^2)$ which satisfies $\sup_{(\theta,\rho,\sigma_1^2,\sigma_2^2) \in \mathcal{S}} | g_n(\theta,\rho,\sigma_1^2,\sigma_2^2)| = O(1)$. For example $\theta \sigma_1/(1-\rho^2) = O_u(1)$.	We also let $O_{up}(1)$ denote any real function $g_n(\theta,\rho,\sigma_1^2,\sigma_2^2,Z_{1,n},Z_{2,n})$ which satisfies $\sup_{(\theta,\rho,\sigma_1^2,\sigma_2^2) \in \mathcal{S}} | g_n(\theta,\rho,\sigma_1^2,\sigma_2^2,Z_{1,n},Z_{2,n}) | = O_p(1)$. For example $z_{1,1} \sigma_1 = O_{up}(1)$.

The following lemma is essential when establishing the asymptotic distribution of the microergodic parameters.
%The proof can be found in the Appendix.

\begin{lem} \label{lem:cross:CLT}
	With the same notations and assumptions as in Theorem \ref{thm:consistency}, let 
	\[L(\theta) = \sum_{i=2}^n \frac{  
		(z_{1,i} - e^{-\theta \Delta_i}z_{1,i-1})
		(z_{2,i} - e^{-\theta \Delta_i}z_{2,i-1})}
	{1-e^{-2 \theta \Delta_i}},\]
	and $G = [\partial/\partial \theta] L(\theta)$.  Let for $n \in \mathbb{N}$ and $i=2,...,n$,
	\begin{equation} \label{eq:def:Yin}
	Y_{i,n} = \frac{\left( z_{1,i} - e^{-\theta_0 \Delta_i} z_{1,i-1} \right) \left( z_{2,i} - e^{-\theta_0 \Delta_i} z_{2,i-1} \right)}{\sigma_{01} \sigma_{02} \sqrt{1+ \rho_0^2} (1-e^{-2 \theta_0 \Delta_i})}.
	\end{equation}
	Then for all $n \in \mathbb{N}$, the $(Y_{i,n})_{i=2,...,n}$ are independent with $\mathbb{E}(Y_{i,n}) = \rho_0 / (1+\rho_0^2)^{1/2}$ and $var(Y_{i,n}) = 1$. Furthermore we have
	\begin{equation*}
	G = -\frac{\sigma_{01} \sigma_{02} \sqrt{1+\rho_0^2} \theta_0}{\theta^2}\sum_{i=2}^n Y_{i,n}
	+ O_{up}(1).
	\end{equation*}
\end{lem}

Using the previous lemma, the following theorem establishes the asymptotic distribution of the MLE of the microergodic parameters.
Specifically we consider three cases: first when both the colocated correlation and variance
parameters are known, second when only the variance parameters are known and third
when all the microergodic parameters are unknown.

\begin{thm} 
	With the same notation and assumptions as in Theorem \ref{thm:consistency}, 
	if $a_{\sigma_{k}}=b_{\sigma_{k}}=\sigma_{0k}^{2}=\hat{\sigma}^{2}_k$ for $k=1,2$, $a_{\rho}=b_{\rho}=\rho_{0}=\hat{\rho}$ and 
	$a_{\theta}< \theta_{0}< b_{\theta}$
	then
	\begin{equation}
	\sqrt{n}(\hat{\theta}-\theta_{0})\stackrel{\mathcal{D}}{\longrightarrow}
	\mathcal{N}
	\left(
	0, \theta_{0}^{2}
	\right).\label{dotto}
	\end{equation}\\
	If $a_{\sigma_{k}}=b_{\sigma_{k}}=\sigma_{0k}^{2}=\hat{\sigma}^{2}_{k}$ for $k=1,2$, $a_{\rho}< \rho_{0}< b_{\rho}$   and $a_{\theta}< \theta_{0}< b_{\theta}$, then\\
	\begin{equation}\label{res2}
	\sqrt{n}\left(\begin{matrix}\hat{\theta}-\theta_{0}\\
	\hat{\rho}-\rho_{0}
	\end{matrix}\right)\stackrel{\mathcal{D}}{\longrightarrow}\mathcal{N}
	\left(
	0,\Sigma_{\theta \rho}
	\right),
	\end{equation}\\
	where $\Sigma_{\theta \rho}=\left(\begin{matrix}\theta_{0}^{2}(1+\rho_{0}^{2})&\theta_{0}\rho_{0}(1-\rho_{0}^{2})\\
	\theta_{0}\rho_{0}(1-\rho_{0}^{2})& (\rho_{0}^{2}-1)^{2} \end{matrix}\right).$\\\\	
	Finally, if $a_{\sigma_{k}}< \sigma_{0k}^{2} < b_{\sigma_{k}}$ for $k=1,2$, $a_{\rho}< \rho_{0}< b_{\rho}$ and $a_{\theta}< \theta_{0}< b_{\theta}$, then
	\begin{equation}\label{res3}
	\sqrt{n}\left(\begin{matrix}\hat{\sigma}_{1}^{2}\hat{\theta}-\sigma_{01}^{2}\theta_{0}\\
	\hat{\sigma}_{2}^{2}\hat{\theta}-\sigma_{02}^{2}\theta_{0}\\
	\hat{\rho}-\rho_{0}
	\end{matrix}\right)\stackrel{\mathcal{D}}{\longrightarrow}\mathcal{N}
	\left(
	0, \Sigma_{f} 
	\right),
	\end{equation}\\	
	where $\Sigma_{f}  =\left(\begin{matrix}2(\theta_{0}\sigma_{01}^{2})^{2}&2(\theta_{0}\rho_{0}\sigma_{01}\sigma_{02})^{2}&
	\theta_{0}\rho_{0}\sigma_{01}^{2}(1-\rho_{0}^{2})\\
	2(\theta_{0}\rho_{0}\sigma_{01}\sigma_{02})^{2}&2(\theta_{0}\sigma_{02}^{2})^{2}&\theta_{0}\rho_{0}\sigma_{02}^{2}(1-\rho_{0}^{2})\\
	\theta_{0}\rho_{0}\sigma_{01}^{2}(1-\rho_{0}^{2})&\theta_{0}\rho_{0}\sigma_{02}^{2}(1-\rho_{0}^{2})&(\rho_{0}^{2}-1)^{2}
	\end{matrix}\right).$\\\\
\end{thm}

\begin{proof}
	Let  $s_x(\psi)=\frac{\partial{}}{\partial x}l_{n}(\psi)$
	the derivative of the negative log-likelihood with respect to $x=\sigma_1^2,\sigma_2^2,\theta,\rho$.
	From Lemma \ref{lem:cross:CLT} and from Equation (3.11) in \cite{Yin1991} we can write, with $W_{k,i,n}$ as in the proof of Theorem \ref{thm:consistency}, 
	\begin{equation}
	s_{\theta}(\psi)
	=\frac{2n}{\theta}-\frac{1}{1-\rho^{2}}\left( \sum_{k=1}^{2}\sum_{i=2}^{n}\frac{\sigma_{0k}^{2}\theta_{0}}{\sigma_{k}^{2}\theta^{2}}W_{k,i,n}^{2} -
	2\rho (1+\rho^{2}_{0})^{\frac{1}{2}}\frac{\sigma_{01}\sigma_{02}\theta_{0}} {\sigma_{1}\sigma_{2}\theta^{2}}\sum_{i=2}^{n} Y_{i,n}
	\right)+O_{up}(1).\label{deri1}
	\end{equation}
	%	where $Y_{i,n}=\frac{(Z_{1,i}-e^{-\theta_{0}\Delta}Z_{1,i-1})(Z_{2,i}-e^{-\theta_{0}\Delta}Z_{2,i-1})}
	%	{\sigma_{01}\sigma_{02}(1+\rho^{2}_{0})^{\frac{1}{2}}(1-e^{-2\theta_{0}\Delta})}$.\\\\
	Then from (\ref{deri1}) we have
	\begin{eqnarray}\label{deri2}
	\theta^{2}(1-\rho^{2})	s_{\theta}(\psi)&=&(n-1) \left[2\theta(1-\rho^{2})-\theta_{0}\left(
	\frac{\sigma^{2}_{01}}{\sigma^{2}_{1}}-2\rho \rho_{0}\frac{\sigma_{01}\sigma_{02}}{\sigma_{1} \sigma_{2}}+\frac{\sigma^{2}_{02}}{\sigma^{2}_{2}} \right) \right]\\&&-	\sum_{k=1}^{2}\sum_{i=2}^{n}\frac{\sigma^{2}_{0k}\theta_{0}}{\sigma^{2}_{k}}\xi_{k,i}
	+2\rho(1+\rho^{2}_{0})^{\frac{1}{2}}\theta_{0}\frac{\sigma_{01}\sigma_{02}}{\sigma_{1}\sigma_{2}}\sum_{i=2}^{n}\xi_{3,i}+O_{up}(1),\nonumber
	\end{eqnarray}
	with $\xi_{k,i}=W^{2}_{k,i,n}-1,\,\,\,k=1,2$ and $\xi_{3,i}=Y_{i,n}-\frac{\rho_{0}}{(1+\rho^{2}_{0})^{\frac{1}{2}}}$.\\\\
	Then $\hat{\psi}$ satisfies $s_{\theta}(\hat{\psi})=0$ and in view of (\ref{deri2}), we get
	\begin{eqnarray}\label{deri3}
	0=\hat{\theta}^{2}(1-\hat{\rho}^{2})s_{\theta}(\hat{\psi})&=&(n-1) \left[2\hat{\theta}(1-\hat{\rho}^{2})-\theta_{0}\left(
	\frac{\sigma^{2}_{01}}{\hat{\sigma}^{2}_{1}}-2\hat{\rho} \rho_{0}\frac{\sigma_{01}\sigma_{02}}{\hat{\sigma}_{1} \hat{\sigma}_{2}}+\frac{\sigma^{2}_{02}}{\hat{\sigma}^{2}_{2}} \right) \right]\\&&-\sum_{k=1}^{2}\sum_{i=2}^{n}\frac{\sigma^{2}_{0k}\theta_{0}}{\hat{\sigma}^{2}_{k}}\xi_{k,i}
	+2\hat{\rho}(1+\rho^{2}_{0})^{\frac{1}{2}}\theta_{0}\frac{\sigma_{01}\sigma_{02}}{\hat{\sigma}_{1}\hat{\sigma}_{2}}\sum_{i=2}^{n}
	\xi_{3,i}+O_{p}(1).\nonumber
	\end{eqnarray}
	If we set $a_{\sigma_{k}}=b_{\sigma_{k}}=\sigma_{0k}^{2}=\hat{\sigma}^{2}_{k}$ for $k=1,2$ and $a_{\rho}=b_{\rho}=\rho_{0}=\hat{\rho}$  in (\ref{deri3}), we get
	\begin{equation}\label{der1}
	0=
	2(n-1)\left[ (\hat{\theta}-\theta_{0})(1-\rho_{0}^{2})\right]
	-\theta_{0}\left[
	\sum_{k=1}^{2}\sum_{i=2}^{n}\xi_{k,i}-
	2\rho_{0}(1+\rho^{2}_{0})^{\frac{1}{2}}\sum_{i=2}^{n}\xi_{3,i}\right]
	+O_{p}(1).
	\end{equation}
	Hence (\ref{der1}) implies
	\begin{equation*}
	\sqrt{n}(\hat{\theta}-\theta_{0})=
	\frac{\theta_{0}n^{-\frac{1}{2}}}{2(1-\rho_{0}^{2})}\left[
	\sum_{k=1}^{2}\sum_{i=2}^{n}\xi_{k,i}-
	2\rho_{0}(1+\rho^{2}_{0})^{\frac{1}{2}}\sum_{i=2}^{n}\xi_{3,i}\right]
	+O_{p}\left(n^{-\frac{1}{2}}\right).
	\end{equation*}\\
	On the other hand, from the multivariate central limit theorem we get
	\begin{equation}\label{lim}
	n^{-\frac{1}{2}}\sum_{i=2}^{n}\left(\begin{matrix}
	\xi_{1,i}\\\xi_{2,i}\\\xi_{3,i}
	\end{matrix}\right)\stackrel{\mathcal{D}}{\longrightarrow}\mathcal{N}\left(0,\Sigma_{\xi}\right),
	\end{equation}
	where
	$\Sigma_{\xi}=\left(\begin{matrix}	
	2 & 2\rho_{0}^{2} & \frac{2\rho_{0}}{(1+\rho_{0}^{2})^{\frac{1}{2}}}\\
	2\rho_{0}^{2}&2&\frac{2\rho_{0}}{(1+\rho_{0}^{2})^{\frac{1}{2}}}\\
	\frac{2\rho_{0}}{(1+\rho_{0}^{2})^{\frac{1}{2}}}&\frac{2\rho_{0}}{(1+\rho_{0}^{2})^{\frac{1}{2}}}&1
	\end{matrix}\right),$ is obtained by calculating $Cov(\xi_{m,i},\xi_{l,i})$ for $m,l=1,2,3$ and $i=2,\ldots,n$.\\
	Hence we have
	\begin{equation*}
	\sqrt{n}(\hat{\theta}-\theta_{0})\stackrel{\mathcal{D}}{\longrightarrow}\mathcal{N}
	\left( 0,\Sigma_{\theta} 
	\right),
	\end{equation*}
	where
	$\Sigma_{\theta} = \frac{\theta_{0}^{2}}{4(1-\rho_{0}^{2})^{2}} \left(\begin{matrix}1&1&-	2\rho_{0}(1+\rho^{2}_{0})^{\frac{1}{2}}\end{matrix}\right) 
	\left(\begin{matrix}	
	2 & 2\rho_{0}^{2} & \frac{2\rho_{0}}{(1+\rho_{0}^{2})^{\frac{1}{2}}}\\
	2\rho_{0}^{2}&2&\frac{2\rho_{0}}{(1+\rho_{0}^{2})^{\frac{1}{2}}}\\
	\frac{2\rho_{0}}{(1+\rho_{0}^{2})^{\frac{1}{2}}}&\frac{2\rho_{0}}{(1+\rho_{0}^{2})^{\frac{1}{2}}}&1
	\end{matrix}\right)\left(\begin{matrix}1\\1\\-2\rho_{0}(1+\rho^{2}_{0})^{\frac{1}{2}}\end{matrix}\right). $\\\\ 
	Then, computing the previous quadratic form, we get 
	\begin{equation*}
	\sqrt{n}(\hat{\theta}-\theta_{0})\stackrel{\mathcal{D}}{\longrightarrow}\mathcal{N}
	\left(
	0, \theta_{0}^{2}
	\right),
	\end{equation*}
	so (\ref{dotto}) is proved. Now, we first prove (\ref{res2}) and (\ref{res3}) for $\rho_{0}\in (-1,1)\diagdown\{0\}$ and discuss the case $\rho_{0}=0$ at the end of the proof.
	To show (\ref{res2}), take differentiation with respec to $\rho$. From the proof of Theorem 2 given in \cite{Yin1991}, and from arguments similar to those of the proof of Lemma \ref{lem:cross:CLT}, we get
	\begin{eqnarray}\nonumber
	s_{\rho}(\psi)&=& \frac{-2 n \rho}{(1-\rho^{2})}+
	\frac{2 \rho}{(1-\rho^{2})^{2}}\left(\sum_{k=1}^{2}\sum_{i=2}^{n}	\frac{\sigma^{2}_{0k}\theta_{0}}{\sigma^{2}_{k}\theta}W^{2}_{k,i,n}
	-\frac{(1+\rho^{2})(1+\rho_{0}^{2})^{\frac{1}{2}}\sigma_{01}\sigma_{02}\theta_{0}}{\rho \sigma_{1}\sigma_{2}\theta}\sum_{i=2}^{n}Y_{i,n}\right)\\&&+O_{up}(1).\label{derr1}
	\end{eqnarray}	
	Then (\ref{derr1}) implies	
	\begin{eqnarray}\label{derr2}
	-\frac{(1-\rho^{2})^{2}\theta}{2\rho}s_{\rho}(\psi)&=&(n-1) \left[\theta(1-\rho^{2})-\theta_{0}\left(\frac{\sigma^{2}_{01}}{\sigma^{2}_{1}}-
	\frac{(1+\rho^{2})\rho_{0}}{\rho}\frac{\sigma_{01}\sigma_{02}}{\sigma_{1}\sigma_{2}}+\frac{\sigma^{2}_{02}}{\sigma^{2}_{2}}  \right)\right]\\\nonumber&&-
	\sum_{k=1}^{2}	\sum_{i=2}^{n}\frac{\sigma^{2}_{0k}\theta_{0}}{\sigma^{2}_{k}}\xi_{k,i}-
	\frac{(1+\rho^{2})(1+\rho^{2}_{0})^{\frac{1}{2}}\theta_{0}}{\rho} \frac{\sigma_{01}\sigma_{02}}{\sigma_{1}\sigma_{2}} \sum_{i=2}^{n}\xi_{3,i}\\&&+O_{up}(1),\nonumber
	\end{eqnarray}
	%with $\xi_{1,i}=W^{2}_{i,n}-1,\,\,\,\xi_{2,i}=V^{2}_{i,n}-1$ and %$\xi_{3,i}=Y_{i,n}-\frac{\rho_{0}}{(1+\rho^{2}_{0})^{\frac{1}{2}}}$.
	Then $\hat{\psi}$ satisfies $s_{\rho}(\hat{\psi})=0$ and in view of (\ref{derr2}), we get
	\begin{eqnarray}\nonumber
	0=-\frac{(1-\hat{\rho}^{2})^{2}\hat{\theta}}{2\hat{\rho}}s_{\rho}(\hat{\psi})&=&(n-1) \left[\hat{\theta}(1-\hat{\rho}^{2})-\theta_{0}\left(\frac{\sigma^{2}_{01}}{\hat{\sigma}^{2}_{1}}-
	\frac{(1+\hat{\rho}^{2})\rho_{0}}{\hat{\rho}}\frac{\sigma_{01}\sigma_{02}}{\hat{\sigma}_{1}\hat{\sigma}_{2}}+\frac{\sigma^{2}_{02}}{\hat{\sigma}^{2}_{2}}  \right)\right]\\\nonumber&&-
	\sum_{k=1}^{2}\sum_{i=2}^{n}\frac{\sigma^{2}_{0k}\theta_{0}}{\hat{\sigma}^{2}_{k}}\xi_{k,i}-
	\frac{(1+\hat{\rho}^{2})(1+\rho^{2}_{0})^{\frac{1}{2}}\theta_{0}}{\hat{\rho}} \frac{\sigma_{01}\sigma_{02}}{\hat{\sigma}_{1}\hat{\sigma}_{2}} \sum_{i=2}^{n}\xi_{3,i}\\&&+O_{p}(1).\label{deri0}
	\end{eqnarray}
	Then we can write, from (\ref{deri3}) and (\ref{deri0})	
	\begin{eqnarray}\nonumber
	\left(\begin{matrix}
	0\\
	0
	\end{matrix}\right)&=&(n-1)\left(\begin{matrix} 2\hat{\theta}(1-\hat{\rho}^{2})-\theta_{0}\left(
	\frac{\sigma^{2}_{01}}{\hat{\sigma}^{2}_{1}}-2\hat{\rho}\rho_{0}\frac{\sigma_{01}\sigma_{02}}{\hat{\sigma}_{1}\hat{\sigma}_{2}}+\frac{\sigma^{2}_{02}}{\hat{\sigma}^{2}_{2}} \right) \\ 
	\hat{\theta}(1-\hat{\rho}^{2})-\theta_{0}\left(\frac{\sigma^{2}_{01}}{\hat{\sigma}^{2}_{1}}-
	\frac{(1+\hat{\rho}^{2})\rho_{0}}{\hat{\rho}}\frac{\sigma_{01}\sigma_{02}}{\hat{\sigma}_{1}\hat{\sigma}_{2}}+\frac{\sigma^{2}_{02}}{\hat{\sigma}^{2}_{2}}  \right)
	\end{matrix}\right)\\&&-
	\left(\begin{matrix}	
	\frac{\sigma^{2}_{01}\theta_{0}}{\hat{\sigma}^{2}_{1}}&
	\frac{\sigma^{2}_{02}\theta_{0}}{\hat{\sigma}^{2}_{2}}&
	-2\hat{\rho}(1+\rho^{2}_{0})^{\frac{1}{2}}\theta_{0}\frac{\sigma_{01}\sigma_{02}}{\hat{\sigma}_{1}\hat{\sigma}_{2}}
	\\
	\frac{\sigma^{2}_{01}\theta_{0}}{\hat{\sigma}^{2}_{1}}&
	\frac{\sigma^{2}_{02}\theta_{0}}{\hat{\sigma}^{2}_{2}}&
	-\frac{(1+\hat{\rho}^{2})(1+\rho^{2}_{0})^{\frac{1}{2}}\theta_{0}}{\hat{\rho}} \frac{\sigma_{01}\sigma_{02}}{\hat{\sigma}_{1}\hat{\sigma}_{2}}\\
	\end{matrix}\right)	\left(\begin{matrix}
	\sum_{i=2}^{n}\xi_{1,i}\\\sum_{i=2}^{n}\xi_{2,i}\\\sum_{i=2}^{n}\xi_{3,i}
	\end{matrix}\right)+O_{p}(1).\label{deriv1} 
	\end{eqnarray}
	If we set $a_{\sigma_{k}}=b_{\sigma_{k}}=\sigma_{0k}^{2}=\hat{\sigma}^{2}_{k}$ for $k=1,2$ in (\ref{deriv1}), we get
	\begin{eqnarray}\nonumber
	\left(\begin{matrix}
	0\\0\end{matrix}\right)&=&(n-1)\left(\begin{matrix}2\hat{\theta}(1-\hat{\rho}^{2})-2\theta_{0}\left(
	1-\hat{\rho}\rho_{0}\right)    \\ 
	\hat{\theta}(1-\hat{\rho}^{2})-\theta_{0}\left(2-
	\frac{(1+\hat{\rho}^{2})\rho_{0}}{\hat{\rho}}  \right) \end{matrix}\right)\\&&	\nonumber
	-\theta_{0}\left(\begin{matrix}	
	1&
	1&
	-2\hat{\rho}(1+\rho^{2}_{0})^{\frac{1}{2}}
	\\
	1&
	1&
	-\frac{(1+\hat{\rho}^{2})(1+\rho^{2}_{0})^{\frac{1}{2}}}{\hat{\rho}} 
	\end{matrix}\right)	\left(\begin{matrix}
	\sum_{i=2}^{n}\xi_{1,i}\\\sum_{i=2}^{n}\xi_{2,i}\\\sum_{i=2}^{n}\xi_{3,i}
	\end{matrix}\right)+O_{p}(1)\\\nonumber
	&=&(n-1)\left(\begin{matrix}2(\hat{\theta}-\theta_{0})-2\hat{\rho}\left(\hat{\theta}\hat{\rho}-\theta_{0}\rho_{0}\right) \\
	(\hat{\theta}-\theta_{0})-\hat{\rho}\left(\hat{\theta}\hat{\rho}-\theta_{0}\rho_{0}\right)-\frac{\theta_{0}}{\hat{\rho}}(\hat{\rho}-\rho_{0})
	\end{matrix}\right)\\&&\label{dis1}
	-\theta_{0}\left[ \left(\begin{matrix}	
	1&
	1&
	-2\rho_{0}(1+\rho^{2}_{0})^{\frac{1}{2}}
	\\
	1&
	1&
	-\frac{(1+\rho^{2}_{0})^{\frac{3}{2}}}{\rho_{0}} 
	\end{matrix}\right)+o_{p}(1) \right]\left(\begin{matrix}
	\sum_{i=2}^{n}\xi_{1,i}\\\sum_{i=2}^{n}\xi_{2,i}\\\sum_{i=2}^{n}\xi_{3,i}
	\end{matrix}\right)+O_{p}(1).
	\end{eqnarray}
	Furthermore,
	\begin{footnotesize}
		\begin{eqnarray}
		\left(\begin{matrix}2(\hat{\theta}-\theta_{0})-2\hat{\rho}\left(\hat{\theta}\hat{\rho}-\theta_{0}\rho_{0}\right) \\
		(\hat{\theta}-\theta_{0})-\hat{\rho}\left(\hat{\theta}\hat{\rho}-\theta_{0}\rho_{0}\right)-\frac{\theta_{0}}{\hat{\rho}}(\hat{\rho}-\rho_{0})
		\end{matrix}\right)
		&=&\left(\begin{matrix}2&-2\hat{\rho}&0\\
		1&-\hat{\rho}&-\frac{\theta_{0}}{\hat{\rho}} \end{matrix}\right)\left(\begin{matrix}\hat{\theta}-\theta_{0}\\
		\hat{\theta}\hat{\rho}-\theta_{0}\rho_{0}\\
		\hat{\rho}-\rho_{0}
		\end{matrix}\right) \nonumber \\
		&=&\left[\left(\begin{matrix}2&-2\rho_{0}&0\\
		1&-\rho_{0}&-\frac{\theta_{0}}{\rho_{0}} \end{matrix}\right)+o_{p}(1)\right]\left(\begin{matrix}\hat{\theta}-\theta_{0}\\
		\rho_{0}(\hat{\theta}-\theta_{0})+\hat{\theta}(\hat{\rho}-\rho_{0})	\\
		\hat{\rho}-\rho_{0}
		\end{matrix}\right) \nonumber \\
		&=&\left[\left(\begin{matrix}2&-2\rho_{0}&0\\
		1&-\rho_{0}&-\frac{\theta_{0}}{\rho_{0}} \end{matrix}\right)+o_{p}(1)\right]
		\left[\left(\begin{matrix} 1&0\\
		\rho_{0}&\hat{\theta}\\
		0&1
		\end{matrix}\right)
		\left(\begin{matrix}\hat{\theta}-\theta_{0}\\
		\hat{\rho}-\rho_{0}
		\end{matrix}\right)\right] \nonumber \\
		&=&
		\left[\left(\begin{matrix}2&-2\rho_{0}&0\\
		1&-\rho_{0}&-\frac{\theta_{0}}{\rho_{0}} \end{matrix}\right)
		\left(\begin{matrix} 1&0\\
		\rho_{0}&\theta_{0}\\
		0&1
		\end{matrix}\right)+o_{p}(1)\right]
		\left(\begin{matrix}\hat{\theta}-\theta_{0}\\
		\hat{\rho}-\rho_{0}
		\end{matrix}\right) \nonumber \\
		&=&
		\left[\left(\begin{matrix}2-2\rho_{0}^{2}&-2\rho_{0}\theta_{0}\\
		1-\rho_{0}^{2}&-\theta_{0}\rho_{0}-\frac{\theta_{0}}{\rho_{0}} \end{matrix}\right)+o_{p}(1)\right]
		\left(\begin{matrix}\hat{\theta}-\theta_{0}\\
		\hat{\rho}-\rho_{0}
		\end{matrix}\right).\nonumber \\ \label{dis2}
		\end{eqnarray}
	\end{footnotesize}
	By taking the inverse of the $2 \times 2$ matrix in (\ref{dis2}), we get from (\ref{dis1}): %using Maple 14,
	\begin{equation*}
	\sqrt{n}
	\left(\begin{matrix}\hat{\theta}-\theta_{0}\\
	\hat{\rho}-\rho_{0}
	\end{matrix}\right) =	
	\theta_{0}n^{-\frac{1}{2}}\left(\begin{matrix}\frac{1}{2}&\frac{1}{2}&0\\
	-\frac{\rho_{0}}{2\theta_{0}}&-\frac{\rho_{0}}{2\theta_{0}}&\frac{\sqrt{1+\rho_{0}^{2}}}{\theta_{0}} \end{matrix}\right)\left(\begin{matrix}
	\sum_{i=2}^{n}\xi_{1,i}\\\sum_{i=2}^{n}\xi_{2,i}\\\sum_{i=2}^{n}\xi_{3,i}
	\end{matrix}\right)+o_{p}(1).
	\end{equation*}\\
	From (\ref{lim}) we can get
	\begin{equation*}
	\sqrt{n}\left(\begin{matrix}\hat{\theta}-\theta_{0}\\
	\hat{\rho}-\rho_{0}
	\end{matrix}\right)\stackrel{\mathcal{D}}{\longrightarrow}\mathcal{N}
	\left(
	0,\Sigma_{\theta\rho} 
	\right),
	\end{equation*}
	where
	$$ \Sigma_{\theta\rho}=\theta_{0}^{2} \left(\begin{matrix}\frac{1}{2}&\frac{1}{2}&0\\
	-\frac{\rho_{0}}{2\theta_{0}}&-\frac{\rho_{0}}{2\theta_{0}}&\frac{\sqrt{1+\rho_{0}^{2}}}{\theta_{0}} \end{matrix}\right)		\left(\begin{matrix}	
	2 & 2\rho_{0}^{2} & \frac{2\rho_{0}}{(1+\rho_{0}^{2})^{\frac{1}{2}}}\\
	2\rho_{0}^{2}&2&\frac{2\rho_{0}}{(1+\rho_{0}^{2})^{\frac{1}{2}}}\\
	\frac{2\rho_{0}}{(1+\rho_{0}^{2})^{\frac{1}{2}}}&\frac{2\rho_{0}}{(1+\rho_{0}^{2})^{\frac{1}{2}}}&1
	\end{matrix}\right)  \left(\begin{matrix}\frac{1}{2}&\frac{1}{2}&0\\
	-\frac{\rho_{0}}{2\theta_{0}}&-\frac{\rho_{0}}{2\theta_{0}}&\frac{\sqrt{1+\rho_{0}^{2}}}{\theta_{0}} \end{matrix}\right)^{\top}.
	$$ 
	Then, we get that\\
	\begin{equation}\label{case2}
	\sqrt{n}\left(\begin{matrix}\hat{\theta}-\theta_{0}\\
	\hat{\rho}-\rho_{0}
	\end{matrix}\right)\xrightarrow[\mathcal{D}]{}\mathcal{N}
	\left(
	0, \left(\begin{matrix}\theta_{0}^{2}(1+\rho_{0}^{2})&\theta_{0}\rho_{0}(1-\rho_{0}^{2})\\
	\theta_{0}\rho_{0}(1-\rho_{0}^{2})& (-1+\rho_{0}^{2})^{2} \end{matrix}\right)
	\right).
	\end{equation}\\\\
	Let us now show \eqref{res3}. Similarly as for \eqref{derr1}, we can show 
	\begin{equation}\label{ders1}
	s_{\sigma_{1}^{2}}(\psi)=\frac{n}{\sigma_{1}^{2}}-
	\frac{\sigma^{2}_{01}\theta_{0}}{(1-\rho^2){\sigma}^{4}_{1}\theta}\sum_{i=2}^{n}W^{2}_{1,i,n}+
	\rho (1+{\rho}_{0}^{2})^{\frac{1}{2}}   \frac{\sigma_{01}\sigma_{02}\theta_{0}}{(1-\rho^2){\sigma}_{1}^{3}{\sigma}_{2}\theta}\sum_{i=2}^{n}Y_{i,n}+O_{up}(1).
	\end{equation}	
	Then (\ref{ders1}) implies
	\begin{eqnarray}\label{ders2}
	\sigma_{1}^{2}(1-\rho^{2})\theta	s_{\sigma_{1}^{2}}(\psi)&=& (n-1)\left[{\theta}(1-{\rho}^{2})-\theta_{0}\left(\frac{\sigma^{2}_{01}}{{\sigma}^{2}_{1}}-{\rho}\rho_{0}\frac{\sigma_{01}\sigma_{02}}{{\sigma}_{1}{\sigma}_{2}}\right)\right]\\&&-
	\frac{\sigma^{2}_{01}\theta_{0}}{{\sigma}^{2}_{1}}\sum_{i=2}^{n}\xi_{1,i}+
	{\rho}(1+{\rho}_{0}^{2})^{\frac{1}{2}}\theta_{0}\frac{\sigma_{01}\sigma_{02}}{{\sigma}_{1}{\sigma}_{2}}\sum_{i=2}^{n}
	\xi_{3,i}+O_{up}(1).\nonumber
	\end{eqnarray}
	%	with $\xi_{1,i}=W^{2}_{i,n}-1$ and % $\xi_{3,i}=Y_{i,n}-\frac{\rho_{0}}{(1+\rho^{2}_{0})^{\frac{1}{2}}}$.\\\\
	Then $\hat{\psi}$ satisfies $s_{\sigma_{1}^{2}}(\hat{\psi})=0$ and in view of (\ref{ders2}), we get
	\begin{eqnarray}\nonumber
	0=\hat{\sigma}_{1}^{2}(1-\hat{\rho}^{2})\hat{\theta}s_{\sigma_{1}^{2}}(\hat{\psi})	&=& (n-1)\left[\hat{\theta}(1-\hat{\rho}^{2})-\theta_{0}\left(\frac{\sigma^{2}_{01}}{\hat{\sigma}^{2}_{1}}-\hat{\rho}\rho_{0}\frac{\sigma_{01}\sigma_{02}}{\hat{\sigma}_{1}\hat{\sigma}_{2}}\right)\right]\\&&-
	\frac{\sigma^{2}_{01}\theta_{0}}{\hat{\sigma}^{2}_{1}}\sum_{i=2}^{n}\xi_{1,i}+
	\hat{\rho}(1+\rho_0^{2})^{\frac{1}{2}}\theta_{0}\frac{\sigma_{01}\sigma_{02}}{\hat{\sigma}_{1}\hat{\sigma}_{2}}\sum_{i=2}^{n}
	\xi_{3,i}+O_{p}(1).\label{deriv10}
	\end{eqnarray}
	Then we can write, from (\ref{deri3}), (\ref{deri0})	and (\ref{deriv10})
	\begin{eqnarray}\nonumber
	\left(\begin{matrix}
	0\\
	0\\
	0	  
	\end{matrix}\right)&=&(n-1)\left(\begin{matrix} 2\hat{\theta}(1-\hat{\rho}^{2})-\theta_{0}\left(
	\frac{\sigma^{2}_{01}}{\hat{\sigma}^{2}_{1}}-2\hat{\rho}\rho_{0}\frac{\sigma_{01}\sigma_{02}}{\hat{\sigma}_{1}\hat{\sigma}_{2}}+\frac{\sigma^{2}_{02}}{\hat{\sigma}^{2}_{2}} \right) \\ 
	\hat{\theta}(1-\hat{\rho}^{2})-\theta_{0}\left(\frac{\sigma^{2}_{01}}{\hat{\sigma}^{2}_{1}}-
	\frac{(1+\hat{\rho}^{2})\rho_{0}}{\hat{\rho}}\frac{\sigma_{01}\sigma_{02}}{\hat{\sigma}_{1}\hat{\sigma}_{2}}+\frac{\sigma^{2}_{02}}{\hat{\sigma}^{2}_{2}}  \right)\\ \hat{\theta}(1-\hat{\rho}^{2})-\theta_{0}\left(\frac{\sigma^{2}_{01}}{\hat{\sigma}^{2}_{1}}-\hat{\rho}\rho_{0}\frac{\sigma_{01}\sigma_{02}}{\hat{\sigma}_{1}\hat{\sigma}_{2}}\right) \end{matrix}\right)\\&&-
	\left(\begin{matrix}	
	\frac{\sigma^{2}_{01}\theta_{0}}{\hat{\sigma}^{2}_{1}}&
	\frac{\sigma^{2}_{02}\theta_{0}}{\hat{\sigma}^{2}_{2}}&
	-2\hat{\rho}(1+\rho^{2}_{0})^{\frac{1}{2}}\theta_{0}\frac{\sigma_{01}\sigma_{02}}{\hat{\sigma}_{1}\hat{\sigma}_{2}}
	\\
	\frac{\sigma^{2}_{01}\theta_{0}}{\hat{\sigma}^{2}_{1}}&
	\frac{\sigma^{2}_{02}\theta_{0}}{\hat{\sigma}^{2}_{2}}&
	-\frac{(1+\hat{\rho}^{2})(1+\rho^{2}_{0})^{\frac{1}{2}}\theta_{0}}{\hat{\rho}} \frac{\sigma_{01}\sigma_{02}}{\hat{\sigma}_{1}\hat{\sigma}_{2}}\\
	\frac{\sigma^{2}_{01}\theta_{0}}{\hat{\sigma}^{2}_{1}}&
	0&
	-\hat{\rho}(1+\rho_{0}^{2})^{\frac{1}{2}}\theta_{0}\frac{\sigma_{01}\sigma_{02}}{\hat{\sigma}_{1}\hat{\sigma}_{2}}
	\end{matrix}\right)	\left(\begin{matrix}
	\sum_{i=2}^{n}\xi_{1,i}\\\sum_{i=2}^{n}\xi_{2,i}\\\sum_{i=2}^{n}\xi_{3,i}
	\end{matrix}\right)+O_{p}(1). \label{case3}
	\end{eqnarray}	\\
	If all parameters are uknown, we get after some  tedious algebra:
	\begin{eqnarray*}
		\left(\begin{matrix} 0\\0\\0\end{matrix}\right)&=&(n-1)\left(\begin{matrix} 
			\frac{1}{\hat{\sigma}_{1}^{2}}&\frac{1}{\hat{\sigma}_{2}^{2}}&
			-\frac{2\hat{\rho}}{\hat{\sigma}_{1}\hat{\sigma_{2}}}\\
			\frac{1}{\hat{\sigma}_{1}^{2}}&\frac{1}{\hat{\sigma}_{2}^{2}}&-\frac{(\hat{\rho}^{2}+1)}{\hat{\rho}\hat{\sigma}_{1}\hat{\sigma_{2}}}\\
			\frac{1}{\hat{\sigma}_{1}^{2}}&0&-\frac{\hat{\rho}}{\hat{\sigma}_{1}\hat{\sigma_{2}}}	
		\end{matrix}\right)\left(\begin{matrix} \hat{\theta}\hat{\sigma}_{1}^{2}-\theta_{0}\sigma_{01}^{2}\\
		\hat{\theta}\hat{\sigma}_{2}^{2}-\theta_{0}\sigma_{02}^{2}\\
		\hat{\theta}\hat{\rho}\hat{\sigma}_{1}\hat{\sigma}_{2}-\theta_{0}\rho_{0}\sigma_{01}\sigma_{02}
	\end{matrix}\right)
	\\&&-
	\left(\begin{matrix}	
		\frac{\sigma^{2}_{01}\theta_{0}}{\hat{\sigma}^{2}_{1}}&
		\frac{\sigma^{2}_{02}\theta_{0}}{\hat{\sigma}^{2}_{2}}&
		-2\hat{\rho}(1+\rho^{2}_{0})^{\frac{1}{2}}\theta_{0}\frac{\sigma_{01}\sigma_{02}}{\hat{\sigma}_{1}\hat{\sigma}_{2}}
		\\
		\frac{\sigma^{2}_{01}\theta_{0}}{\hat{\sigma}^{2}_{1}}&
		\frac{\sigma^{2}_{02}\theta_{0}}{\hat{\sigma}^{2}_{2}}&
		-\frac{(1+\hat{\rho}^{2})(1+\rho^{2}_{0})^{\frac{1}{2}}\theta_{0}}{\hat{\rho}} \frac{\sigma_{01}\sigma_{02}}{\hat{\sigma}_{1}\hat{\sigma}_{2}}\\
		\frac{\sigma^{2}_{01}\theta_{0}}{\hat{\sigma}^{2}_{1}}&
		0&
		-\hat{\rho}(1+\rho_{0}^{2})^{\frac{1}{2}}\theta_{0}\frac{\sigma_{01}\sigma_{02}}{\hat{\sigma}_{1}\hat{\sigma}_{2}}
	\end{matrix}\right)	\left(\begin{matrix}
	\sum_{i=2}^{n}\xi_{1,i}\\\sum_{i=2}^{n}\xi_{2,i}\\\sum_{i=2}^{n}\xi_{3,i}
\end{matrix}\right)+O_{p}(1). 
\end{eqnarray*}	
Applying LU matrix factorization we get
\begin{eqnarray*}
	\left(\begin{matrix} 0\\0\\0\end{matrix}\right)&=&(n-1)
	\left(\begin{matrix}	
		\frac{1}{\hat{\sigma}_{1}\hat{\sigma}_{2}}& 0 & 0\\
		\frac{1}{\hat{\sigma}_{1}\hat{\sigma}_{2}}& \frac{1}{\hat{\sigma}_{1}\hat{\sigma}_{2}}& 0\\
		\frac{1}{\hat{\sigma}_{1}\hat{\sigma}_{2}} & 0 &\frac{1}{\hat{\sigma}_{1}\hat{\sigma}_{2}}	
	\end{matrix}\right)	
	\left(\begin{matrix} \frac{\hat{\sigma}_{2}}{\hat{\sigma}_{1}}& \frac{\hat{\sigma}_{1}}{\hat{\sigma}_{2}}&-2\hat{\rho}\\
		0 & 0& \frac{\hat{\rho}^{2}-1}{\hat{\rho}}\\
		0 & -\frac{\hat{\sigma}_{1}}{\hat{\sigma}_{2}}&\hat{\rho}
	\end{matrix}\right)
	\left(\begin{matrix} \hat{\theta}\hat{\sigma}_{1}^{2}-\theta_{0}\sigma_{01}^{2}\\
		\hat{\theta}\hat{\sigma}_{2}^{2}-\theta_{0}\sigma_{02}^{2}\\
		\hat{\theta}\hat{\rho}\hat{\sigma}_{1}\hat{\sigma}_{2}-\theta_{0}\rho_{0}\sigma_{01}\sigma_{02}
	\end{matrix}\right)
	\\&&-\theta_{0}\left(\begin{matrix}	
		\frac{1}{\hat{\sigma}_{1}\hat{\sigma}_{2}}& 0 & 0\\
		\frac{1}{\hat{\sigma}_{1}\hat{\sigma}_{2}}& \frac{1}{\hat{\sigma}_{1}\hat{\sigma}_{2}}& 0\\
		\frac{1}{\hat{\sigma}_{1}\hat{\sigma}_{2}} & 0 &\frac{1}{\hat{\sigma}_{1}\hat{\sigma}_{2}}	
	\end{matrix}\right)	
	\left(\begin{matrix}
		\sigma_{01}^{2}\frac{\hat{\sigma}_{2}}{\hat{\sigma}_{1}}&\sigma_{02}^{2}\frac{\hat{\sigma}_{1}}{\hat{\sigma}_{2}}&-2(1+\rho_{0}^{2})^{\frac{1}{2}}\sigma_{01}\sigma_{02}\hat{\rho}\\
		0 & 0 & (1+\rho_{0}^{2})^{\frac{1}{2}}\sigma_{01}\sigma_{02} \frac{\hat{\rho}^{2}-1}{\hat{\rho}}\\
		0 & -\sigma_{02}^{2}\frac{\hat{\sigma}_{1}}{\hat{\sigma}_{2}}& (1+\rho_{0}^{2})^{\frac{1}{2}}\sigma_{01}\sigma_{02}\hat{\rho}
	\end{matrix}\right)
	\left(\begin{matrix}
		\sum_{i=2}^{n}\xi_{1,i}\\\sum_{i=2}^{n}\xi_{2,i}\\\sum_{i=2}^{n}\xi_{3,i}
	\end{matrix}\right)\\&&+O_{p}(1). 
\end{eqnarray*}	
Hence we get\\
\begin{eqnarray}\nonumber
\left(\begin{matrix} 0\\0\\0\end{matrix}\right)&=&(n-1)	
\left(
\left(\begin{matrix} \frac{\sigma_{02}}{\sigma_{01}}& \frac{\sigma_{01}}{\sigma_{02}}&-2\rho_{0}\\
0 & 0& \frac{\rho_{0}^{2}-1}{\rho_{0}}\\
0 & -\frac{\sigma_{01}}{\sigma_{02}}&\rho_{0}
\end{matrix}\right)
+o_p(1)
\right)
\left(\begin{matrix} \hat{\theta}\hat{\sigma}_{1}^{2}-\theta_{0}\sigma_{01}^{2}\\
\hat{\theta}\hat{\sigma}_{2}^{2}-\theta_{0}\sigma_{02}^{2}\\
\hat{\theta}\hat{\rho}\hat{\sigma}_{1}\hat{\sigma}_{2}-\theta_{0}\rho_{0}\sigma_{01}\sigma_{02}
\end{matrix}\right)
\\ \nonumber&&-\theta_{0}
\left(
\left(\begin{matrix}
\sigma_{01}\sigma_{02}&\sigma_{01}\sigma_{02}&
-2\rho_{0}(1+\rho_{0}^{2})^{\frac{1}{2}}\sigma_{01}\sigma_{02}\\
0 & 0 & \frac{(\rho_{0}^{2}-1)(1+\rho_{0}^{2})^{\frac{1}{2}}}{\rho_{0}}\sigma_{01}\sigma_{02} \\
0 & -\sigma_{01}\sigma_{02}  & \rho_{0}(1+\rho_{0}^{2})^{\frac{1}{2}}\sigma_{01}\sigma_{02}
\end{matrix}\right)
+o_p(1)
\right)
\left(\begin{matrix}
\sum_{i=2}^{n}\xi_{1,i}\\\sum_{i=2}^{n}\xi_{2,i}\\\sum_{i=2}^{n}\xi_{3,i}
\end{matrix}\right)\\&&+O_{p}(1). \label{eq13}
\end{eqnarray}	
Furthermore, we have
\begin{equation}\label{inv}
\left(\begin{matrix} \frac{\sigma_{02}}{\sigma_{01}}& \frac{\sigma_{01}}{\sigma_{02}}&-2\rho_{0}\\
0 & 0& \frac{\rho_{0}^{2}-1}{\rho_{0}}\\
0 & -\frac{\sigma_{01}}{\sigma_{02}}&\rho_{0}
\end{matrix}\right)^{-1}\left(\begin{matrix}
\sigma_{01}\sigma_{02}&\sigma_{01}\sigma_{02}&
-2\rho_{0}(1+\rho_{0}^{2})^{\frac{1}{2}}\sigma_{01}\sigma_{02}\\
0 & 0 & \frac{(\rho_{0}^{2}-1)(1+\rho_{0}^{2})^{\frac{1}{2}}}{\rho_{0}}\sigma_{01}\sigma_{02} \\
0 & -\sigma_{01}\sigma_{02}  & \rho_{0}(1+\rho_{0}^{2})^{\frac{1}{2}}\sigma_{01}\sigma_{02}
\end{matrix}\right)=\left(\begin{matrix} \sigma_{01}^{2}& 0 & 0\\
0 & \sigma_{02}^{2}& 0\\
0 & 0& \sqrt{\rho_{0}^{2}+1} \sigma_{01} \sigma_{02}
\end{matrix}\right).
\end{equation}
Hence, from (\ref{inv}) and (\ref{eq13}), we obtain \\
\begin{equation*}
\sqrt{n} \left(\begin{matrix} \hat{\theta}\hat{\sigma}_{1}^{2}-\theta_{0}\sigma_{01}^{2}\\
\hat{\theta}\hat{\sigma}_{2}^{2}-\theta_{0}\sigma_{02}^{2}\\
\hat{\theta}\hat{\rho}\hat{\sigma}_{1}\hat{\sigma}_{2}-\theta_{0}\rho_{0}\sigma_{01}\sigma_{02}
\end{matrix}\right) =  \theta_{0}n^{-\frac{1}{2}} \left(\begin{matrix} \sigma_{01}^{2}& 0 & 0\\
0 & \sigma_{02}^{2}& 0\\
0 & 0& \sqrt{\rho_{0}^{2}+1} \sigma_{01} \sigma_{02}
\end{matrix}\right)
\left(\begin{matrix}
\sum_{i=2}^{n}\xi_{1,i}\\\sum_{i=2}^{n}\xi_{2,i}\\\sum_{i=2}^{n}\xi_{3,i}
\end{matrix}\right)+o_{p}(1).
\end{equation*}\\\\
Hence from (\ref{lim}) we can get
\begin{equation*}
\sqrt{n}\left(\begin{matrix} \hat{\theta}\hat{\sigma}_{1}^{2}-\theta_{0}\sigma_{01}^{2}\\
\hat{\theta}\hat{\sigma}_{2}^{2}-\theta_{0}\sigma_{02}^{2}\\
\hat{\theta}\hat{\rho}\hat{\sigma}_{1}\hat{\sigma}_{2}-\theta_{0}\rho_{0}\sigma_{01}\sigma_{02}
\end{matrix}\right)\stackrel{\mathcal{D}}{\longrightarrow}\mathcal{N}
\left(
0,\Sigma_{\theta\rho\sigma_{1}\sigma_{2}} 
\right),
\end{equation*}
where\\
$\Sigma_{\theta\rho\sigma_{1}\sigma_{2}}= \theta_{0}^{2}\left(\begin{matrix} \sigma_{01}^{2}& 0 & 0\\
0 & \sigma_{02}^{2}& 0\\
0 & 0& \sqrt{\rho_{0}^{2}+1} \sigma_{01} \sigma_{02}
\end{matrix}\right)		\left(\begin{matrix}	
2 & 2\rho_{0}^{2} & \frac{2\rho_{0}}{(1+\rho_{0}^{2})^{\frac{1}{2}}}\\
2\rho_{0}^{2}&2&\frac{2\rho_{0}}{(1+\rho_{0}^{2})^{\frac{1}{2}}}\\
\frac{2\rho_{0}}{(1+\rho_{0}^{2})^{\frac{1}{2}}}&\frac{2\rho_{0}}{(1+\rho_{0}^{2})^{\frac{1}{2}}}&1
\end{matrix}\right)  \left(\begin{matrix} \sigma_{01}^{2}& 0 & 0\\
0 & \sigma_{02}^{2}& 0\\
0 & 0& \sqrt{\rho_{0}^{2}+1} \sigma_{01} \sigma_{02}
\end{matrix}\right)^{\top}.$ \\\\
Then, we get:
$$\Sigma_{\theta\rho\sigma_{1}\sigma_{2}}=\left(\begin{matrix}2(\theta_{0}\sigma_{01}^{2})^{2} & 2(\theta_{0}\rho_{0}\sigma_{01}\sigma_{02})^{2}&
2\theta_{0}^{2}\rho_{0}\sigma_{01}^{3}\sigma_{02}\\
2(\theta_{0}\rho_{0}\sigma_{01}\sigma_{02})^{2}&2(\theta_{0}\sigma_{02}^{2})^{2}&
2\theta_{0}^{2}\rho_{0}\sigma_{02}^{3}\sigma_{01}\\
2\theta_{0}^{2}\rho_{0}\sigma_{01}^{3}\sigma_{02}&
2\theta_{0}^{2}\rho_{0}\sigma_{02}^{3}\sigma_{01}&\theta_{0}^{2}(\rho_{0}^{2}+1)^{2}\sigma_{01}^{2}\sigma_{02}^{2}
\end{matrix}\right).$$\\
Let 
$\,\,f \left( \begin{matrix} \hat{\theta}\hat{\sigma}_{1}^{2}\\
\hat{\theta}\hat{\sigma}_{2}^{2}\\
\hat{\theta}\hat{\rho}\hat{\sigma}_{1}\hat{\sigma}_{2}\end{matrix} \right)= \left( \begin{matrix} \hat{\theta}\hat{\sigma}_{1}^{2}\\
\hat{\theta}\hat{\sigma}_{2}^{2}\\
\frac{\hat{\theta}\hat{\rho}\hat{\sigma}_{1}\hat{\sigma}_{2}}{\sqrt{\hat{\theta}\hat{\sigma}_{1}^{2}}\sqrt{\hat{\theta}\hat{\sigma}_{2}^{2}}}  \end{matrix} \right).$\\\\\\
Then, using the multivariate Delta Method we get
\begin{equation*}
\sqrt{n}
\left(\begin{matrix} \hat{\theta}\hat{\sigma}_{1}^{2}-\theta_{0}\sigma_{01}^{2}\\
\hat{\theta}\hat{\sigma}_{2}^{2}-\theta_{0}\sigma_{02}^{2}\\
\hat{\rho}-\rho_{0}
\end{matrix}\right)
\stackrel{\mathcal{D}}{\longrightarrow}\mathcal{N}
\left(
0, \Sigma_{f}
\right),
\end{equation*}\\
where 
$ \Sigma_{f}=H_{f} \Sigma_{\theta\rho\sigma_{1}\sigma_{2}}H_{f}^{\top}$ and $H_{f}=\left(\begin{matrix}
1 & 0 & 0\\
0 & 1 & 0\\ 
-\frac{\rho_{0}}{2\sigma_{01}^{2}\theta_{0}}& -\frac{\rho_{0}}{2\sigma_{02}^{2}\theta_{0}} &
\frac{1}{\sigma_{01}\sigma_{02}\theta_{0}}
\end{matrix}\right).$\\
Finally, we get
\begin{equation}\label{case3f}
\sqrt{n}\left(\begin{matrix}\hat{\sigma}_{1}^{2}\hat{\theta}-\sigma_{01}^{2}\theta_{0}\\
\hat{\sigma}_{2}^{2}\hat{\theta}-\sigma_{02}^{2}\theta_{0}\\
\hat{\rho}-\rho_{0}
\end{matrix}\right)\stackrel{\mathcal{D}}{\longrightarrow}\mathcal{N}
\left(
0, \left(\begin{matrix}2(\theta_{0}\sigma_{01}^{2})^{2}&2(\theta_{0}\rho_{0}\sigma_{01}\sigma_{02})^{2}&
\theta_{0}\rho_{0}\sigma_{01}^{2}(1-\rho_{0}^{2})\\
2(\theta_{0}\rho_{0}\sigma_{01}\sigma_{02})^{2}&2(\theta_{0}\sigma_{02}^{2})^{2}&\theta_{0}\rho_{0}\sigma_{02}^{2}(1-\rho_{0}^{2})\\
\theta_{0}\rho_{0}\sigma_{01}^{2}(1-\rho_{0}^{2})&\theta_{0}\rho_{0}\sigma_{02}^{2}(1-\rho_{0}^{2})&(\rho_{0}^{2}-1)^{2}
\end{matrix}\right)
\right).
\end{equation}\\
In the case $\rho_{0}=0$, we can show that (\ref{case2}) and (\ref{case3f}) are still true with the same proof. The only difference is that we multiply the second line of (\ref{deriv1}) and (\ref{case3}) by $\hat{\rho}$. We skip the technical details.
\end{proof}

\section{Numerical experiments} \label{sec:simulation}

The main goal of this section  is to  compare the finite sample behavior of the MLE  of the covariance parameters of  model (\ref{exp:cov})
with the asymptotic distribution given in Section \ref{sec:norm}.
We consider two possible scenarios for our simulation study:
\begin{enumerate}
	\item The variances parameters are known and we estimate jointly $\rho_0$ and $\theta_0$.
	\item We estimate jointly all the parameters $\sigma_{01}^2$, $\sigma_{02}^2$,  $\rho_0$ and $\theta_0$.
\end{enumerate}

Under the  first  scenario
we simulate, using the Cholesky decomposition, $1000$ realizations from
a bivariate zero mean stochastic process with covariance model (\ref{exp:cov}) observed on $n=200, 500$ points uniformly distributed in $[0,1]$.
We simulate fixing $\sigma_{01}^2=\sigma_{02}^2=1$ and  increasing values for the colocated correlation parameter 
and the scale parameter, that is 
$\rho_0=0, 0.2, 0.5$ and $\theta_0=3/x$ with $x=0.2, 0.4, 0.6$.  Note that $\theta_0$ is 
is parametrized in terms of practical range that is the correlation is lower than $0.05$ when the distance between the points is greater than $x$.
For each simulated realization, we compute $\hat{\rho}_i$ and $\hat{\theta}_i$, $i=1,\ldots,1000$, i.e.  the MLE of the colocated correlation and scale parameters.
Using the asymptotic  distribution given   in Equation (\ref{res2}), Tables \ref{tab1}, \ref{tab2} compare the empirical quantiles of 
order $0.05, 0.25, 0.5, 0.75, 0.95$ of $[\sqrt{n}(\hat{\theta}_i-\theta_0)/ \sqrt{\theta_0^2(1+\rho_0^2)}]_{i_=1}^{1000}$
and $[\sqrt{n}(\hat{\rho}_i-\rho_0)/ \sqrt{(\rho_0^2-1)^2}]_{i_=1}^{1000}$ respectively, 
with the theoretical quantiles of the standard Gaussian distribution when  $n=200,500$.
The  simulated variances of  $\hat{\rho}_i$ and $\hat{\theta}_i$  for $i=1,\ldots,1000$ are also reported. 

As a general comment, it can be noted that the asymptotic approximation given in Equation  (\ref{res2}) improves
and the variances of the MLE of $\rho_0$ and $\theta_0$ decrease when increasing $n$ from $200$ to $500$.
When $n=500$ the asymptotic approximation works  very well.

%quite satisfactory
%with the exception of the case $\rho=0.5$ 
%where some problems of convergence on the tails of the distributions can be noted,
%in particular when  $\theta=3/0.4,3/0.6$.

Under the second scenario we set $\sigma_{01}^2 =\sigma_{02}^2=0.5$ and the other parameters as in Scenario 1.
In this case  we simulate, using Cholesky decomposition, $1000$ realizations from
a bivariate zero mean stochastic process with covariance model (\ref{exp:cov}) observed on $n=500, 1000$ points uniformly distributed in $[0,1]$.
For each simulated realization, we obtain $\hat{\sigma}^2_{1i}$, $\hat{\sigma}^2_{2i}$
$\hat{\rho}_i$ and $\hat{\theta}_i$, $i=1,\ldots,1000$ the MLE of  the two  variances, the colocated correlation and scale parameters.
Using the asymptotic  distribution given in  Equation (\ref{res3}), Tables \ref{tab3}, \ref{tab4}, \ref{tab5} compare the empirical quantiles of 
order $0.05, 0.25, 0.5, 0.75, 0.95$ of  $[\sqrt{n}(\hat{\sigma}^2_{1i}\hat{\theta}_i- \sigma^2_{01}\theta_0 )/ \sqrt{2(\sigma^2_{01}\theta_0)^2}]_{i_=1}^{1000}$, $[\sqrt{n}(\hat{\sigma}^2_{2i}\hat{\theta}_i- \sigma^2_{02}\theta_0 )/ \sqrt{2(\sigma^2_{02}/\theta_0)^2}]_{i_=1}^{1000}$ and $[\sqrt{n}(\hat{\rho}_i-\rho_0)/ \sqrt{(\rho_0^2-1)^2}]_{i_=1}^{1000}$ respectively, for $n=500, 1000$
with the  theoretical  quantiles of the standard Gaussian distribution.
The  simulated variances of   $\hat{\sigma}^2_{1i}\hat{\theta}_i$ ,  $\hat{\sigma}^2_{2i}\hat{\theta}_i$ and $\hat{\rho}_i$ and for $i=1,\ldots,1000$ are also reported. 
As in the previous Scenario,  the asymptotic approximation given in Equation  (\ref{res3}) improves
and the variances of the MLE of $\rho_0$ and $\sigma_{0i}^2\theta_0$, $i=1,2$ reduce when increasing $n$ from $500$ to $1000$.
When $n=1000$ the asymptotic approximation is quite satisfactory,
with the exception of the case $\rho_0=0.5$ 
where some problems of convergence on the tails of the distributions can be noted,
in particular when  $\theta_0=3/0.4,3/0.6$.\\\\

\begin{table}
	\begin{footnotesize}
		%$\hat{\rho}$\\
		\begin{tabular}{|c| c| c| c| c| c| c| c| c| }
			\hline
			$n$&  $\theta_0$ & $\rho_0$  & 5$\%$ & 25$\%$ & 50$\%$ & 75$\%$ & 95$\%$ & Var \\ 
			\hline \hline
			200 & 3/0.2& 0 & -1.6070 & -0.6521& -0.0335&  0.6812&  1.7225& 0.0051 \\
			500 & 3/0.2 & 0 &-1.6416 & -0.6255&  0.0022&  0.6675&  1.6499& 0.0019 \\
			\hline
			200 & 3/0.2 & 0.2 & -1.6755 & -0.6749 & -0.0161 &  0.7149 & 1.6455  &  0.0048\\
			500 & 3/0.2 & 0.2 & -1.6336 &-0.6786 & -0.0113 & 0.6712 & 1.6361& 0.0018 \\
			\hline
			200 & 3/0.2 & 0.5 & -1.7768 &-0.6809 & -0.0232 &  0.6583 & 1.6119 &  0.0030 \\
			500 & 3/0.2 & 0.5 & -1.6586 & -0.6490 &  0.0146 &  0.6321 &  1.6709 & 0.0011 \\
			\hline
			200 & 3/0.4 & 0&-1.6185& -0.6531& -0.0292&  0.6852& 1.7259&  0.0051   \\
			500 & 3/0.4 & 0&-1.6454 &-0.6248& -0.0029&  0.6616&  1.6457& 0.0019 \\
			\hline
			200 & 3/0.4 & 0.2& -1.6781 &-0.6688 & -0.0031 & 0.7142 & 1.6576 &  0.0048  \\
			500 & 3/0.4 & 0.2& -1.6291 & -0.6750 & -0.0059 &  0.6755 &  1.6629 & 0.0018 \\
			\hline
			200 & 3/0.4 & 0.5& -1.7716 & -0.6874 & -0.0282 &  0.6580 & 1.6226 & 0.0030 \\
			500 & 3/0.4 & 0.5&  -1.6436 & -0.6534 &  0.0082 &  0.6270 &  1.6788 & 0.0011 \\
			\hline
			200 & 3/0.6 & 0 & -1.6179& -0.6554& -0.0288&  0.6845&  1.7200& 0.0051 \\
			500 & 3/0.6 & 0&  -1.6487& -0.6466& -0.0019&  0.6645&  1.6513& 0.0019  \\
			\hline 
			200 & 3/0.6 & 0.2 & -1.6908 &-0.6694 &-0.0088 & 0.7120 & 1.6681 & 0.0048\\
			500 & 3/0.6 & 0.2 & -1.6286 & -0.6767 & -0.0111 &  0.6704 &  1.6608 & 0.0018\\
			\hline
			200 & 3/0.6 & 0.5 &  -1.7810 & -0.6950 & -0.0354 &  0.6642 & 1.6121 & 0.0030\\
			500 & 3/0.6 & 0.5 &  -1.6407 & -0.6537 &  0.0073 &  0.6255 &  1.6686 & 0.0011\\
			\hline
			\multicolumn{3}{|c|}{$\mathcal{N} (0,1)$} &-1.6448 & -0.6744 &  0 &  0.6744 & 1.6448&\\
			\hline 
		\end{tabular}
	\end{footnotesize}
	\caption{For scenario 1: empirical quantiles, and variances of simulated MLE of $\rho_0$ for different values of $\rho_0$ and $\theta_0$, when $n=200, 500$.}\label{tab1}
\end{table}

\begin{table}
	\begin{footnotesize}
		%$\hat{\theta}$\\
		\begin{tabular}{|c| c| c| c| c| c| c| c| c| }
			\hline
			$n$&   $\theta_0$& $\rho_0$  & 5$\%$ & 25$\%$ & 50$\%$ & 75$\%$ & 95$\%$ & Var \\ 
			\hline \hline
			200 & 3/0.2 & 0 & -1.6567& -0.7382& -0.0978&  0.6805&  1.7761& 2.50e-05  \\
			500 & 3/0.2 & 0 & -1.6838& -0.7447& -0.0469&  0.6684&  1.6369 & 9.23e-06\\
			\hline
			200 & 3/0.2  & 0.2 & -1.6176 & -0.7432 & -0.0651&  0.6583&  1.8583 & 2.61e-05\\
			500 & 3/0.2 & 0.2 & -1.6962 &-0.7370 &-0.0260 &  0.6533 & 1.6414 & 9.61e-06\\
			\hline
			200 & 3/0.2 & 0.5 & -1.6032 &-0.7028 &-0.0725 & 0.6689 &  1.8607 & 3.12e-05 \\
			500 & 3/0.2 & 0.5 & -1.6530 & -0.7169 & -0.0600 &  0.6758 &  1.6320 & 1.13e-05 \\
			\hline
			200 & 3/0.4 & 0& -1.5910& -0.7551& -0.0907&  0.6715&  1.8092& 9.68e-05  \\
			500 & 3/0.4 & 0& -1.6852& -0.7522& -0.0367&  0.6661&  1.6850& 3.64e-05\\
			\hline
			200 & 3/0.4 & 0.2& -1.6073 &-0.7242 &-0.0731 &  0.6261&  1.7977 & 1.01e-04 \\
			500 & 3/0.4 & 0.2& -1.6841 & -0.7469 & -0.0217 &  0.6649 &  1.6060 & 3.79e-05\\
			\hline
			200 & 3/0.4 & 0.5&   -1.5561 & -0.6992 &-0.0599 & 0.6578 & 1.8200 & 1.02e-04 \\
			500 & 3/0.4 & 0.5&  -1.6410 & -0.7191 & -0.0577 &  0.6772 & 1.6024 &  4.48e-05\\
			\hline
			200 & 3/0.6 & 0 & -1.5563& -0.7307& -0.0847&  0.6711&  1.8093& 2.15e-04\\
			500 & 3/0.6 & 0 & -1.6737& -0.7421& -0.0352&  0.6635&  1.6752&  8.16e-05\\
			\hline 
			200 & 3/0.6 & 0.2 & -1.5693 &-0.7187 &-0.0694 & 0.6130 & 1.8244 & 2.01e-04\\
			500 & 3/0.6 & 0.2 &  -1.6821 & -0.7473 & -0.0373 &  0.6579 &  1.6315 & 8.49e-05\\
			\hline
			200 & 3/0.6 & 0.5 &  -1.5666 & -0.6765 & -0.0638 &  0.6659 &  1.8175 & 2.05e-04\\
			500 & 3/0.6 & 0.5 & -1.6373 & -0.7232 & -0.0566 &  0.6669 &  1.6208 & 1.03e-04\\
			\hline
			\multicolumn{3}{|c|}{$\mathcal{N} (0,1)$} &-1.6448 & -0.6744 &  0 &  0.6744 & 1.6448&\\
			\hline 
		\end{tabular}
	\end{footnotesize}
	\caption{For scenario 1: empirical quantiles, and variances of simulated MLE of $\theta_0$ for different values of $\rho_0$ and $\theta_0$, when $n=500, 1000$.}\label{tab2}
\end{table}

\begin{table}
	\begin{footnotesize}
		\begin{tabular}{|c| c| c| c| c| c| c| c| c| }
			\hline
			$n$&   $\theta_0$& $\rho_0$  & 5$\%$ & 25$\%$ & 50$\%$ & 75$\%$ & 95$\%$ & Var \\ 
			\hline \hline
			500 & 3/0.2  & 0& -1.4333& -0.5971&  0.0547&  0.7163&  1.7152& 0.2100  \\
			1000 & 3/0.2  & 0& -1.6085& -0.6291&  0.0338&  0.7331&  1.65266& 0.1102\\
			\hline
			500 & 3/0.2  & 0.2 & -1.4331& -0.5964&  0.0535&  0.7160&  1.7142& 0.2106  \\
			1000 & 3/0.2  & 0.2 &-1.6022& -0.6257&  0.0356&  0.7348&  1.6526& 0.1095 \\
			\hline
			500 & 3/0.2  & 0.5 & -1.4333& -0.5945&  0.0520& 0.7163&  1.7151& 0.2098\\
			1000 & 3/0.2  & 0.5 & -1.6115& -0.6327&  0.0336&  0.7339&  1.6501&  0.1110 \\
			\hline
			500 & 3/0.4 & 0& -1.4277& -0.5827&  0.0427&  0.6999&  1.6847& 0.0519 \\
			1000 & 3/0.4  & 0& -1.6158& -0.6364&  0.0370&  0.7277&  1.6263& 0.0275\\
			\hline
			500 & 3/0.4  & 0.2& -1.4276& -0.5799&  0.0427&  0.6999&  1.6844& 0.0518 \\
			1000 & 3/0.4  & 0.2& -1.6109& -0.6299&  0.0459&  0.7412&  1.6357& 0.0276  \\
			\hline
			500 & 3/0.4  & 0.5&  -1.4276 & -0.5827&  0.0387&  0.6938&  1.6842& 0.0517  \\
			1000 & 3/0.4  & 0.5& -1.6090& -0.6275&  0.0380&  0.7402&  1.6346& 0.0275 \\
			\hline
			500 & 3/0.6 & 0&  -1.4229& -0.5847&  0.0406&  0.6995&  1.6997& 0.0228 \\
			1000 & 3/0.6 & 0 & -1.6241& -0.6314&  0.0393&  0.7411&  1.6377& 0.0123  \\
			\hline 
			500 & 3/0.6 & 0.2 & -1.4235& -0.5833&  0.0433&  0.7090&  1.6999& 0.0228\\
			1000 & 3/0.6 & 0.2 & -1.6234& -0.6318&  0.0343&  0.7377&  1.6365& 0.0123 \\
			\hline
			500 & 3/0.6 & 0.5 &  -1.4235& -0.5833& 0.0433 &  0.7090&  1.6999 & 0.0228\\
			1000 &3/0.6 & 0.5 & -1.6234& -0.6318&  0.0343&  0.7377&  1.6365&  0.0123\\
			\hline
			\multicolumn{3}{|c|}{$\mathcal{N} (0,1)$} &-1.6448 & -0.6744 &  0 &  0.6744 & 1.6448&\\
			\hline 
		\end{tabular}
	\end{footnotesize}
	\caption{For scenario 2: empirical quantiles, and variances of simulated MLE of $\sigma^2_{01}\theta_0$ for different values of $\rho_0$ and $\theta_0$, when $n=500, 1000$.}\label{tab3}
\end{table}

\begin{table}
	\begin{footnotesize}
		\begin{tabular}{|c| c| c| c| c| c| c| c| c| }
			\hline
			$n$&  $\theta_0$ & $\rho_0$  & 5$\%$ & 25$\%$ & 50$\%$ & 75$\%$ & 95$\%$ & Var \\ 
			\hline \hline
			500 &  3/0.2 & 0 & -1.5318& -0.6282&  0.0544&  0.7382&  1.8544& 0.2336  \\
			1000 &  3/0.2 & 0 & -1.5134& -0.6382&  0.0628&  0.7003&  1.7527& 0.1150 \\
			\hline
			500 & 3/0.2 & 0.2 & -1.5067&-0.6272&  0.0411&  0.7359&  1.7854 & 0.2364\\
			1000 & 3/0.2 & 0.2 & -1.4653& -0.6415&  0.0728&  0.7239&  1.7743& 0.1155\\
			\hline
			500 &  3/0.2 & 0.5 &  -1.4734& -0.6078&  0.0308&  0.7732&  1.8493& 0.2336 \\
			1000 &  3/0.2 & 0.5 & -1.4260& -0.6438&  0.0192&  0.7809&  1.7520& 0.1149\\
			\hline
			500 & 3/0.4 & 0& -1.5173&-0.6479& 0.0598&  0.7225&  1.8452& 0.0578\\
			1000 & 3/0.4 & 0& -1.5014& -0.6395& 0.0604&  0.6989&  1.7377 & 0.0287\\
			\hline
			500 & 3/0.4& 0.2& -1.5164& -0.6275&  0.0553&  0.7537&  1.7436 & 0.0580  \\
			1000 & 3/0.4 & 0.2& -1.4724& -0.6442&  0.0494&  0.7260&  1.7822& 0.0288 \\
			\hline
			500 & 3/0.4& 0.5&  -1.4877 & -0.6099 &  0.0252 &  0.7725 &  1.7729 &  0.0581   \\
			1000 & 3/0.4 & 0.5& -1.4488 &-0.6495&  0.0117&  0.7565&  1.7381&  0.0287 \\
			\hline
			500 & 3/0.6& 0 & -1.5448& -0.6447&  0.0705&  0.7226&  1.8264& 0.0257  \\
			1000 & 3/0.6 & 0 & -1.4940& -0.6560& 0.0548& 0.7055&  1.7365& 0.0128 \\
			\hline 
			500 & 3/0.6& 0.2 & -1.5122& -0.6379&  0.0668&  0.7553&  1.7310& 0.0257\\
			1000 & 3/0.6 & 0.2 & -1.4466& -0.6450&  0.0541&  0.7316&  1.7923& 0.0128\\
			\hline
			500 & 3/0.6& 0.5 & -1.4768& -0.6128&  0.0325&  0.7605&  1.7396 & 0.0258 \\
			1000 & 3/0.6 & 0.5 &-1.4464& -0.6549& -0.0115&  0.7551&  1.7464&  0.0128  \\
			\hline
			\multicolumn{3}{|c|}{$\mathcal{N} (0,1)$} &-1.6448 & -0.6744 &  0 &  0.6744 & 1.6448&\\
			\hline 
		\end{tabular}
	\end{footnotesize}
	\caption{For scenario 2: empirical quantiles, and variances of simulated MLE of $\sigma^2_{02}\theta_0$ for different values of $\rho_0$ and $\theta_0$, when $n=500, 1000$.}\label{tab4}
\end{table}

\begin{table}
	\begin{footnotesize}
		\begin{tabular}{|c| c| c| c| c| c| c| c| c| }
			\hline
			$n$&  $\theta_0$ & $\rho_0$  & 5$\%$ & 25$\%$ & 50$\%$ & 75$\%$ & 95$\%$ & Var \\ 
			\hline \hline
			500 & 3/0.2 & 0 & -1.6477& -0.6271&  0.0016&  0.6795&  1.6786& 0.0019  \\
			1000 & 3/0.2 & 0& -1.7235& -0.6167&  0.0516&  0.6975&  1.7051 & 0.0010\\
			\hline
			500 &  3/0.2 & 0.2 & -1.6431& -0.6714&  0.0037&  0.6518&  1.6418& 0.0018\\
			1000 & 3/0.2 & 0.2 &-1.6620& -0.5992&  0.0460&  0.6906&  1.6757& 0.0009  \\
			\hline
			500 &  3/0.2& 0.5 & -1.6193& -0.6434&  0.0123&  0.6220&  1.6585& 0.0011\\
			1000 & 3/0.2 & 0.5 & -1.6582& -0.6445&  0.0563&  0.6729&  1.5996& 0.0005  \\
			\hline
			500 & 3/0.4& 0& -1.6486& -0.6283& -0.0091&  0.6684&  1.6600& 0.0019 \\
			1000 & 3/0.4 & 0& -1.7296& -0.6209&  0.0365&  0.6967&  1.7151& 0.0010 \\	
			\hline
			500 & 3/0.4& 0.2& -1.6407& -0.6589& -0.0074&  0.6509&  1.6631& 0.0018  \\
			1000 & 3/0.4 & 0.2& -1.6840& -0.6067&  0.0253&  0.6845&  1.6823& 0.0009\\	
			\hline
			500 & 3/0.4& 0.5& -1.6160 &-0.6529& -0.0045&  0.5987&  1.6543& 0.0010\\
			1000 & 3/0.4 & 0.5& -1.6669& -0.6434&  0.0577&  0.6734&  1.6171& 0.0005\\	
			\hline
			500 & 3/0.6& 0 &  -1.6504& -0.6280& -0.0092&  0.6890&  1.6550 & 0.0019\\
			1000 & 3/0.6 & 0 & -1.7330& -0.6214& 0.0370&  0.6931&  1.7297 & 0.0010 \\	
			\hline 
			500 & 3/0.6& 0.2 & -1.6412& -0.6525&  0.0050&  0.6653&  1.6603& 0.0018 \\
			1000 & 3/0.6 & 0.2 & -1.7102& -0.6111&  0.0201&  0.6738&  1.6908& 0.0009 \\	
			\hline
			500 & 3/0.6& 0.5 & -1.6536& -0.6510&  0.0070&  0.6169&  1.6561& 0.0011 \\
			1000 & 3/0.6 & 0.5 & -1.6776& -0.6496&  0.0617&  0.6714&  1.6175&  0.0005   \\	
			\hline
			\multicolumn{3}{|c|}{$\mathcal{N} (0,1)$} &-1.6448 & -0.6744 &  0 &  0.6744 & 1.6448&\\
			\hline 
		\end{tabular}
	\end{footnotesize}
	\caption{For scenario 2: empirical quantiles, and variances of simulated MLE of  $\rho_0$ for different values of $\rho_0$ and $\theta_0$, when $n=500, 1000$.}\label{tab5}
\end{table}

\section{Concluding remarks}
In this paper we considered the fixed domain asymptotic properties of the MLE for a bivariate zero mean Gaussian process
with a separable exponential covariance model. We characterized the equivalence of Gaussian measures under this model
and we established the consistency and the asymptotic distribution of the MLE of the microergodic parameters.
Analogue results under increasing domain asymptotics are obtained by \cite{BevValVel2015}.
It is interesting to note that the asymptotic distribution of the MLE of the colocated correlation parameter, between the two processes,
does not depend on the asymptotic framework.

Our results can be extended in different directions. 
Let ${\cal M}(h,\nu,\theta)=  \frac{2^{1-\nu}}{\Gamma(\nu)} \left (||h||\theta\right )^{\nu} {\cal K}_{\nu} \left (||h||\theta  \right )$,
$h\in \mathbb{R}^d$, $\nu, \theta>0$, be the Matérn correlation model.
A generalization of the bivariate covariance model   (\ref{exp:cov}) is then the following model:
$$
Cov(Z_{i}(s),Z_{j}(s+h);\psi)=\sigma_{i}\sigma_{j}(\rho +(1-\rho){\bf 1}_{i=j}) {\cal M}(h,\nu,\theta_{ij}),\quad i,j=1,2,
$$
with $\theta_{12} = \theta_{21}$, $\sigma_{1}>0$, $\sigma_{2}>0$, where in this case  ${\psi}=(\sigma_{1}^2,\sigma_{2}^2,\theta_{11},\theta_{12},\theta_{22},\nu, \rho)^{\top}$.
This is a special case of the bivariate Matérn model proposed in \cite{GneKleSch2010}. The authors give necessary and sufficient conditions
in terms of $\psi$ for the validity of this kind of  model. Studying the asymptotic properties of the MLE of $\psi$ would then be interesting. The main challenges in this case are the number of parameters involved and the fact that the covariance matrix cannot be factorized as a kronecker product.
Moreover for $\nu\neq0.5$ the markovian property of the process cannot be exploited.

Another interesting extension is to consider the  fixed domain asymptotic properties of the
tapered maximum likelihood estimator in bivariate covariance models. This method of estimation has been proposed
as a possible surrogate for the MLE when working with large data sets, see \cite{Furr:Gent:Nych:06,Kauf:Sche:Nych:08}. Asymptotic properties of this estimator, under fixed domain asymptotics and 
in the univariate case, can be found in \cite{KauSchNyc2008}, \cite{WanLoh2011} and \cite{DuZhaMan2009}. Extensions of these results to the bivariate case would be interesting.
Both topics are to be investigated in future research.

%\section*{Acknowledgement}
%The research work conducted by Moreno Bevilacqua  was supported in part  by
%FONDECYT grant  1160280, Chile.

\section*{Appendix}
{\bf\textit{Proof of lemma \ref{lem:expression:likelihood}.}}\\ 

Let $\Sigma (\psi)=A\otimes R$, where the matrices $A$ and $R$ are defined in (\ref{eq:covmat}).
First, using properties of the determinant  of the Kroneker product, we have:
$$\log|\Sigma(\psi)|= log ( \left| A \right|^{n}  \left|R\right|^{2})=n \log \left[\sigma_{1}^{2} \sigma_{2}^{2} (1-\rho^{2})\right]+2\log \left|R\right|.
$$
From lemma 1 in \cite{Yin1993}, $\left| R \right|= \prod_{i=2}^{n} \left(1-e^{-2\theta
	\Delta_{i}}\right)$. Then, we have
\begin{equation}\label{adet}
\log|\Sigma(\psi)|=
n \log \left[\sigma_{1}^{2} \sigma_{2}^{2} (1-\rho^{2})\right]+2
\sum_{i=2}^{n} \log \left(1-e^{-2\theta \Delta_{i}}\right).
\end{equation}
On the other hand, since $\Sigma (\psi)^{-1}=A^{-1}\otimes R^{-1}$, we obtain
\begin{eqnarray*}
	Z_n^{\top} \left[ \Sigma (\psi)\right]^{-1} Z_n
	&=& \left[ Z_{1,n}^{\top},Z_{2,n}^{\top}\right]\left[\left(
	\begin{array}{cc}\frac{1}{\sigma_{1}^{2}(1-\rho^{2})}R^{-1}&-\frac{\rho}{\sigma_{1}\sigma_{2}(1-\rho^{2})}R^{-1}\\
		-\frac{\rho}{\sigma_{1}\sigma_{2}(1-\rho^{2})}R^{-1}&\frac{1}{\sigma_{2}^{2}(1-\rho^{2})}R^{-1}
	\end{array}\right) \right]\left[\begin{array}{c}
	Z_{1,n}\\Z_{2,n}\end{array} \right]\\
&=&\frac{1}{(1-\rho^{2})}\left\{\frac{1}{\sigma_{1}^{2}}
Z_{1,n}^{\top}R^{-1}Z_{1,n}+\frac{1}{\sigma_{2}^{2}}
Z_{2,n}^{\top}R^{-1}Z_{2,n} \right.\\ &&\,\,\,\,\,\,\,\,\,\,\,\,\,\,\,\,\,\,\,\,\,\,\,\,\,\,\,\,\,\,\left.-\frac{\rho}{\sigma_{1}\sigma_{2}}\left(
Z_{2,n}^{\top}R^{-1}Z_{1,n}+Z_{1,n}^{\top}R^{-1}Z_{2,n}\right) \right\}.
\end{eqnarray*}
Then using Lemma 1 in \cite{Yin1993} (Eq. 4.2) we obtain:
\begin{eqnarray}
Z_n^{\top} \left[ \Sigma (\psi)\right]^{-1}
Z_n&=&\frac{1}{(1-\rho^{2})}\left\{\sum_{k=1}^{2}\frac{1}{\sigma_{k}^{2}}
\left(z_{k,1}^{2}+\sum_{i=2}^{n} \frac{\left(z_{k,i}-e^{-\theta
		\Delta_{i}}z_{k,i-1}\right)^{2}}{1-e^{-2\theta\Delta_{i}}}\right)\right.\label{efcom} \\
&&-\left.
\frac{2\rho}{\sigma_{1}\sigma_{2}}\left(z_{1,1}z_{2,1}+\sum_{i=2}^{n} \frac{\left(z_{1,i}-e^{-\theta
		\Delta_{i}}z_{1,i-1}\right)\left(z_{2,i}-e^{-\theta
		\Delta_{i}}z_{2,i-1}\right)}{1-e^{-2\theta \Delta_{i}}}
\right) \right\}.\nonumber 
\end{eqnarray}

Combining (\ref{log}), (\ref{adet}) and (\ref{efcom}), we obtain
\begin{eqnarray*}
	l_{n}(\psi)&=&n \left[\log (2\pi)+\log(1-\rho^{2})\right]+\sum_{k=1}^{2} \log
	(\sigma_{k}^{2}) + \sum_{k=1}^{2} \sum_{i=2}^{n} \log
	\left[\sigma_{k}^{2}\left(1-e^{-2\theta \Delta_{i} }\right)\right]
	\\
	&&+\frac{1}{1-\rho^{2}}\left\{ \sum_{k=1}^{2}  \frac{1}{\sigma_{k}^{2}}
	\left(z_{k,1}^{2}+\sum_{i=2}^{n} \frac{\left(z_{k,i}-e^{-\theta
			\Delta_{i}}z_{k,i-1}\right)^{2}}{1-e^{-2\theta \Delta_{i}}}
	\right)\right.\\
	&&\,\,\,\,\,\,\,\,\,\,\,\,\,\,\,\,\,\,\,\,\,\,\,\,\,\,\,\,\,-\left.
	\frac{2\rho}{\sigma_{1}\sigma_{2}}\left(z_{1,1}z_{2,1}+\sum_{i=2}^{n} \frac{\left(z_{1,i}-e^{-\theta
			\Delta_{i}}z_{1,i-1}\right)\left(z_{2,i}-e^{-\theta
			\Delta_{i}}z_{2,i-1}\right)}{1-e^{-2\theta \Delta_{i}}}
	\right) \right\}.
\end{eqnarray*}
{\bf \textit{Proof of lemma \ref{lem:cross:CLT}.}}\\ By differentiation of $L(\theta)$ with respect to $\theta$ we obtain
\begin{eqnarray*}
	G & = & \sum_{i=2}^n \frac{
		\Delta_i e^{-\theta \Delta_i}z_{1,i-1} (z_{2,i} - e^{-\theta \Delta_i}z_{2,i-1})
		+(z_{1,i} - e^{-\theta \Delta_i}z_{1,i-1})  \Delta_i e^{-\theta \Delta_i}z_{2,i-1}}
	{1-e^{-2 \theta \Delta_i}} \\
	&  & -\sum_{i=2}^n \frac{ (z_{1,i} - e^{-\theta \Delta_i}z_{1,i-1})
		(z_{2,i} - e^{-\theta \Delta_i}z_{2,i-1})
		2 \Delta_i e^{-2 \theta \Delta_i}
	}{(1-e^{-2 \theta \Delta_i})^2.}
	= G_1 - G_2,
\end{eqnarray*}
say. Let us first show that $G_1 = O_{up}(1)$. Let for $i=2,...,n$, $A_{\theta,i} =  \Delta_i e^{- \theta \Delta_i}/(1-e^{-2 \theta \Delta_i})$. By symmetry of $Z_{1,n}$ and $Z_{2,n}$, in order to show $G_1 = O_{up}(1)$, it is sufficient to show that
\begin{equation} \label{eq:for:Gun}
\sum_{i=2}^n A_{\theta,i}
z_{1,i-1} (z_{2,i} - e^{-\theta \Delta_i}z_{2,i-1}) = O_{up}(1).
\end{equation}
We have
\begin{eqnarray*} \label{eq:intro:Tun:Tdeux}
	\sum_{i=2}^n 
	A_{\theta,i} z_{1,i-1} (z_{2,i} - e^{-\theta \Delta_i}z_{2,i-1}) & = &
	\sum_{i=2}^n A_{\theta,i} z_{1,i-1} (z_{2,i} - e^{-\theta_0 \Delta_i}z_{2,i-1}) \nonumber \\
	& & 
	+
	\sum_{i=2}^n 
	A_{\theta,i} z_{1,i-1}z_{2,i-1}  ( e^{-\theta_0 \Delta_i} - e^{-\theta \Delta_i}) \nonumber \\
	& = & T_1 + T_2,
\end{eqnarray*}
say. Now, one can see from Taylor expansions, and since $\theta \in \Theta$ with $\Theta$ compact in $(0,\infty)$, that
\[	S:= \sup_{\theta \in \Theta} \sup_{n \in \mathbb{N}, i=2,...,n} \left| \frac{A_{\theta,i} (e^{- \theta_0 \Delta_i} - e^{- \theta \Delta_i})}{\Delta_i}\right|  < \infty.\]
Hence
\begin{eqnarray*}
	| T_2 |  & \leq & \sup_{t \in [0,1]} | Z_1(t)Z_2(t) |
	S \sum_{i=1}^n \Delta_i \\
	& = & O_{up}(1).
\end{eqnarray*} 
Let us now consider $T_1$. We have, for any $k<i$
\begin{flalign} \label{eq:for:Tun}
& \mathbb{E} \left(\{ z_{1,i-1} (z_{2,i} - e^{-\theta_0 \Delta_i}z_{2,i-1})
z_{1,k-1} (z_{2,k} - e^{-\theta_0 \Delta_k}z_{2,k-1}) \right\}
& \nonumber \\
& = \mathbb{E} \left\{ 
\mathbb{E} \left[ z_{1,i-1} (z_{2,i} - e^{-\theta_0 \Delta_i}z_{2,i-1})
z_{1,k-1} (z_{2,k} - e^{-\theta_0 \Delta_k}z_{2,k-1}) |
z_{1,1},...,z_{1,i-1},z_{2,1},...,z_{2,i-1}
\right]
\right\} & \nonumber \\
& = \mathbb{E} \left\{  z_{1,i-1} z_{1,k-1} (z_{2,k} - e^{-\theta_0 \Delta_k}z_{2,k-1}) 
\mathbb{E} \left[ (z_{2,i} - e^{-\theta_0 \Delta_i}z_{2,i-1})
|
z_{1,1},...,z_{1,i-1},z_{2,1},...,z_{2,i-1}
\right]
\right\}. &
\end{flalign}
Let us show that $\mathbb{E} \left[ z_{2,i}
|z_{1,1},...,z_{1,i-1},z_{2,1},...,z_{2,i-1}\right]
=  e^{-\theta_0 \Delta_i}z_{2,i-1}$.
Let $r$ be the $1\times(i-1)$ vector $(e^{-(s_i-s_1) \theta_0 },e^{-(s_i-s_2) \theta_0 },...,e^{- (s_i-s_{i-1}) \theta_0 })^{\top}$, let $R = [e^{-|s_a-s_b| \theta_0 })]_{a,b=1 }^{i-1}$ and let $V_k = (z_{k,1},...,z_{k,i-1})^{\top}$ for $k=1,2$. Then
\begin{flalign*}
& \mathbb{E} \left[ z_{2,i}|z_{1,1},...,z_{1,i-1},z_{2,1},...,z_{2,i-1}\right] & \\
& = \mathbb{E} \left[ z_{2,i}|z_{1,1}/\sigma_{01},...,z_{1,i-1}/\sigma_{01},z_{2,1}/\sigma_{02},...,z_{2,i-1}/\sigma_{02}\right] & \\
& =  [\rho_0 \sigma_{02} r^{\top}, \sigma_{02} r^{\top}]\left[
\begin{pmatrix}
1 & \rho_0 \\
\rho_0 & 1\end{pmatrix}^{-1}\otimes R^{-1}\right]
\left[\begin{matrix}
(1/\sigma_{01}) V_1 \\
(1/\sigma_{02}) V_2
\end{matrix}\right] & \\& = [\rho_0 \sigma_{02} r^{\top}, \sigma_{02} r^{\top}]
\left[\begin{matrix}
\frac{1}{1- \rho_0^2} R^{-1} & \frac{- \rho_0}{1- \rho_0^2} R^{-1} \\
\frac{- \rho_0}{1- \rho_0^2} R^{-1} &  \frac{1}{1- \rho_0^2} R^{-1}
\end{matrix}
\right]
\left[\begin{matrix}(1/\sigma_{01}) V_1 \\
(1/\sigma_{02}) V_2
\end{matrix}
\right] & \\
& = \frac{1}{1- \rho_0^2} \left( \rho_0 \sigma_{02} r^{\top} R^{-1} V_1/\sigma_{01} - \rho_0^2 \sigma_{02} r^{\top} R^{-1} V_2/\sigma_{02} - \rho_0 \sigma_{02} r^{\top} R^{-1} V_1/\sigma_{01} + \sigma_{02} r^{\top} R^{-1} V_2 /\sigma_{02} \right) & \\
& =  r^{\top} R^{-1} V_2 . & 
\end{flalign*}
Now, it is well known from the Markovian property of $Z_2$ that $r^{\top} R^{-1} V_2 = e^{- \theta_0 \Delta_i} z_{2,i-1}$. Hence, we have $\mathbb{E} \left[ z_{2,i}|
z_{1,1},...,z_{1,i-1},z_{2,1},...,z_{2,i-1}\right]
=  e^{-\theta_0 \Delta_i}z_{2,i-1}$, which together with $\eqref{eq:for:Tun}$ gives  
\[\mathbb{E} \left(\{ z_{1,i-1} (z_{2,i} - e^{-\theta_0 \Delta_i}z_{2,i-1})z_{1,k-1} (z_{2,k} - e^{-\theta_0 \Delta_k}z_{2,k-1}) \right\}=0 \]
for $k <i$. Hence
\begin{eqnarray} \label{eq:after:proving:decorrelation}
\mathbb{E} \left( T_1^2 \right) & = &
\mathbb{E} \left( \left[
\sum_{i=2}^n 
A_{\theta,i} z_{1,i-1} (z_{2,i} - e^{-\theta_0 \Delta_i}z_{2,i-1}) \right]^2 \right) \nonumber \\
& = &  
\sum_{i=2}^n 
\mathbb{E} \left(
A_{\theta,i}^2 z_{1,i-1}^2 (z_{2,i} - e^{-\theta_0 \Delta_i}z_{2,i-1})^2\right). 
\end{eqnarray}
Now, one can see from Taylor expansions that
\[	S' := \sup_{\theta \in \Theta} \sup_{n \in \mathbb{N}, i=2,...,n} \left| A_{\theta,i} \right|  < \infty.	\]
Hence
\begin{eqnarray*}
	\mathbb{E} \left( T_1^2 \right)
	& \leq &
	S' \sum_{i=2}^n 
	\sqrt{ \mathbb{E} \left(
		z_{1,i-1}^4 \right) }
	\sqrt{ \mathbb{E} \left(
		(z_{2,i} - e^{-\theta_0 \Delta_i}z_{2,i-1})^4\right) } \\
	& = & S' \sum_{i=2}^n \sqrt{3} \sigma_{01}^2 \sqrt{3} \sigma_{02}^2 (1-e^{-2 \theta_0 \Delta_i}).
\end{eqnarray*}
One can see that
\[	S'' := \sup_{\theta \in \Theta} \sup_{n \in \mathbb{N}, i=2,...,n} \left| \frac{(1-e^{-2 \theta_0 \Delta_i})}{\Delta_i} 
\right|  < \infty.\]
Hence, 
\begin{equation} \label{eq:finishing:Tun}
\mathbb{E} \left( T_1^2 \right) \leq S' S'' 3 \sigma_{01}^2 \sigma_{02}^2 \sum_{i=2}^n \Delta_i = O_p(1).
\end{equation}
Hence $T_1 = 0_{up}(1)$ and \eqref{eq:for:Gun} is proved. Hence, we have, with
\[	B_{\theta,i} = \frac{2 \Delta_i e^{- 2 \theta \Delta_i}}{(1-e^{- 2 \theta \Delta_i})^2},
\]
\begin{eqnarray} \label{eq:for:G}
G & = & O_{up}(1) - \sum_{i=2}^n (z_{1,i} - e^{-\theta \Delta_i}z_{1,i-1})
(z_{2,i} - e^{-\theta \Delta_i}z_{2,i-1})B_{\theta,i}.
\end{eqnarray}
Furthermore, using $a_{\theta} b_{\theta} c_{\theta} = a_{\theta_0} b_{\theta_0} c_{\theta} - a_{\theta_0} b_{\theta_0} c_{\theta} + a_{\theta_0} b_{\theta} c_{\theta} - a_{\theta_0} b_{\theta} c_{\theta} + a_{\theta} b_{\theta} c_{\theta}$, we have
\begin{flalign} \label{eq:atheta:btheta}
& \sum_{i=2}^n(z_{1,i} - e^{-\theta \Delta_i}z_{1,i-1})(z_{2,i} - e^{-\theta \Delta_i}z_{2,i-1}) B_{\theta,i} & \nonumber \\
& = \sum_{i=2}^n(z_{1,i} - e^{-\theta_0 \Delta_i}z_{1,i-1})(z_{2,i} - e^{-\theta_0 \Delta_i}z_{2,i-1}) B_{\theta,i}
& \nonumber \\
& \,\,\,\, + 
\sum_{i=2}^n(z_{1,i} - e^{-\theta_0 \Delta_i}z_{1,i-1})(e^{-\theta_0 \Delta_i} z_{2,i-1} - e^{-\theta \Delta_i}z_{2,i-1}) B_{\theta,i}
& \nonumber \\
&\,\,\,\, +\sum_{i=2}^n(e^{-\theta_0 \Delta_i} z_{1,i-1} - e^{-\theta \Delta_i}z_{1,i-1})(z_{2,i} - e^{-\theta \Delta_i}z_{2,i-1}) B_{\theta,i}
& \nonumber \\
& = \sum_{i=2}^n(z_{1,i} - e^{-\theta_0 \Delta_i}z_{1,i-1})(z_{2,i} - e^{-\theta_0 \Delta_i}z_{2,i-1})B_{\theta,i}
+ R_1 + R_2, &
\end{flalign}
say. We now show that $R_1,R_2 = O_{up}(1)$. For $R_1$, we have
\begin{eqnarray*}
	R_1  & = & \sum_{i=2}^n
	(z_{1,i} - e^{-\theta_0 \Delta_i}z_{1,i-1})
	(e^{-\theta_0 \Delta_i} z_{2,i-1} - e^{-\theta \Delta_i}z_{2,i-1}) B_{\theta,i}
	\\
	& = &
	\sum_{i=2}^n
	z_{2,i-1} (z_{1,i} - e^{-\theta_0 \Delta_i}z_{1,i-1})
	(e^{-\theta_0 \Delta_i}  - e^{-\theta \Delta_i}) B_{\theta,i}.
\end{eqnarray*}
As for $T_1$ in \eqref{eq:after:proving:decorrelation},
\[
\mathbb{E}(R_1^2) = 
\sum_{i=2}^n
\mathbb{E} \left( z_{2,i-1}^2 (z_{1,i} - e^{-\theta_0 \Delta_i}z_{1,i-1})^2 \right)
(e^{-\theta_0 \Delta_i}  - e^{-\theta \Delta_i})^2 B_{\theta,i}^2.
\]
One can show using Taylor expansions that 
\[
S^{(3)} := 
\sup_{\theta \in \Theta} \sup_{n \in \mathbb{N}, i=2,...,n} \left| B_{\theta,i}^2 (e^{- \theta_0 \Delta_i} - e^{- \theta \Delta_i})^2
\right|  < \infty.
\]
Hence
\[
\mathbb{E}(R_1^2) \leq 
S^{(3)} \sum_{i=2}^n
\mathbb{E} \left( z_{2,i-1}^2 (z_{1,i} - e^{-\theta_0 \Delta_i}z_{1,i-1})^2 \right) = O_u(1)
\]
as for \eqref{eq:finishing:Tun}. Hence
$R_1 = O_{up}(1)$. For $R_2$, we have
\begin{eqnarray} \label{eq:for:reusing:Gun}
R_2 & = & 
\sum_{i=2}^n
(e^{-\theta_0 \Delta_i} z_{1,i-1} - e^{-\theta \Delta_i}z_{1,i-1})
(z_{2,i} - e^{-\theta \Delta_i}z_{2,i-1})
B_{\theta,i}
\nonumber \\
& = & \sum_{i=2}^n
B_{\theta,i}
(e^{-\theta_0 \Delta_i}  - e^{-\theta \Delta_i})
z_{1,i-1}
(z_{2,i} - e^{-\theta \Delta_i}z_{2,i-1}) \nonumber \\
& = & \sum_{i=2}^n
C_{\theta,i}z_{1,i-1}
(z_{2,i} - e^{-\theta \Delta_i}z_{2,i-1}),
\end{eqnarray}say. We can thus show that $R_2 = O_{up}(1)$ as for \eqref{eq:for:Gun}. Indeed, the only difference between \eqref{eq:for:reusing:Gun} and \eqref{eq:for:Gun} is that $A_{\theta,i}$ is replaced by $C_{\theta,i}$. To show \eqref{eq:for:Gun} we only used that
\[
\sup_{\theta \in \Theta} \sup_{n \in \mathbb{N}, i=2,...,n} \left| A_{\theta,i} 
\right|  < \infty.
\] 
We can see from Taylor expansions that
\[
\sup_{\theta \in \Theta} \sup_{n \in \mathbb{N}, i=2,...,n} \left| C_{\theta,i}
\right|  < \infty.
\] 
Hence, as for \eqref{eq:for:Gun}, we can show that $R_2 = O_{up}(1)$. Hence, from \eqref{eq:for:G} and \eqref{eq:atheta:btheta}, we have,
\begin{equation} \label{eq:for:G:deux}
G = O_{up}(1) - 
\sum_{i=2}^n
(z_{1,i} - e^{-\theta_0 \Delta_i}z_{1,i-1})
(z_{2,i} - e^{-\theta_0 \Delta_i}z_{2,i-1})
B_{\theta,i}.
\end{equation}
Let, for $i=2,...,n$,
\[
X_i =   
(z_{1,i} - e^{-\theta_0 \Delta_i}z_{1,i-1})
(z_{2,i} - e^{-\theta_0 \Delta_i}z_{2,i-1})
B_{\theta,i}.
\]
For $k<i$ we have
\begin{flalign*}
& \mathbb{E} \left( (z_{1,i} - e^{-\theta_0 \Delta_i}z_{1,i-1})
(z_{2,k} - e^{-\theta_0 \Delta_k}z_{2,k-1}) \right) & \\
&  =  \rho_0 \sigma_{01} \sigma_{02}
\left(  e^{- (s_i-s_k) \theta_0 } - e^{-\theta_0 \Delta_i} e^{-(s_{i-1}-s_k)\theta_0 } -  e^{-\theta_0 \Delta_k} e^{-(s_i-s_{k-1})\theta_0 } +  e^{ - \theta_0 (\Delta_i + \Delta_k)} e^{-(s_{i-1}-s_{k-1})\theta_0 }   \right) & \\
& =  0. &
\end{flalign*}
Hence, for $k<i$ (and for $k \neq i$ by symmetry), the random variables
$(z_{1,i} - e^{-\theta_0 \Delta_i}z_{1,i-1})$ and $(z_{2,k} - e^{-\theta_0 \Delta_k}z_{2,k-1})$ are independent. In addition, the random variables
$(z_{j,i} - e^{-\theta_0 \Delta_i}z_{j,i-1})$
and $(z_{j,k} - e^{-\theta_0 \Delta_k}z_{j,k-1}) $
are also independent for $j =1,2$ and $k \neq i$. Hence, the $n-1$ Gaussian vectors $\left\{ \left[ (z_{1,i} - e^{-\theta_0 \Delta_i}z_{1,i-1}),(z_{2,i} - e^{-\theta_0 \Delta_i}z_{2,i-1}) \right] \right\}_{i=2,...,n}$
are mutually independent. Thus, the $\{ X_i \}_{i=2,...,n}$ are independent random variables.\\ 
We also have
\begin{eqnarray*}
	\sum_{i=2}^n X_i & = & \sum_{i=2}^n \left( z_{1,i} - e^{-\theta_0 \Delta_i} z_{1,i-1} \right) \left( z_{2,i} - e^{-\theta_0 \Delta_i} z_{2,i-1} \right) \frac{2 \Delta_i e^{-2 \theta \Delta_i}}{(1-e^{-2 \theta \Delta_i})^2} \\
	& = & \sum_{i=2}^n
	\frac{
		\left( z_{1,i} - e^{-\theta_0 \Delta_i} z_{1,i-1} \right) \left( z_{2,i} - e^{-\theta_0 \Delta_i} z_{2,i-1} \right)
	}
	{
		\sigma_{01} \sigma_{02} \sqrt{1+ \rho_0^2} (1-e^{-2 \theta_0 \Delta_i})
	}
	\frac{
		\sigma_{01} \sigma_{02} \sqrt{1+ \rho_0^2} (1-e^{-2 \theta_0 \Delta_i})
		2 \Delta_i e^{-2 \theta \Delta_i}
	}
	{
		(1-e^{-2 \theta \Delta_i})^2
	}.
\end{eqnarray*}	Let
\[
D_{\theta,i} = 
\frac{
	\sigma_{01} \sigma_{02} \sqrt{1+ \rho_0^2} (1-e^{-2 \theta_0 \Delta_i})
	2 \Delta_i e^{-2 \theta \Delta_i}
}
{
	(1-e^{-2 \theta \Delta_i})^2
}
\]
and let $Y_{i,n}$ be as in \eqref{eq:def:Yin}.
Then, let
\begin{eqnarray*}
	T & = & \left| 
	\sum_{i=2}^n X_i
	- 
	\left( \frac{\sigma_{01} \sigma_{02} \sqrt{1+\rho_0^2} \theta_0}{\theta^2} \right)
	\sum_{i=2}^n Y_{i,n} 
	\right| \\
	& = &
	\left|
	\sum_{i=2}^n Y_{i,n}
	\left( D_{\theta,i} - \frac{\sigma_{01} \sigma_{02} \sqrt{1+\rho_0^2} \theta_0}{\theta^2} \right)
	\right|.
\end{eqnarray*}
On the other hand 
\begin{eqnarray*}
	\mathbb{E}(Y_{i,n}) &=&\frac{\mathbb{E}(z_{1,i}z_{2,i})-e^{-\theta_{0}\Delta_{i}}\mathbb{E}(z_{1,i}z_{2,i-1})
		-e^{-\theta_{0}\Delta_{i}}\mathbb{E}(z_{1,i-1}z_{2,i})
		+e^{-2\theta_{0}\Delta_{i}}\mathbb{E}(z_{1,i-1}z_{2,i-1})  }{\sigma_{01}\sigma_{02}(1+\rho_{0}^{2})^{1/2}(1-e^{-2\theta_{0}\Delta_{i}})}
	\\
	&=&\frac{\sigma_{01}\sigma_{02}\rho_{0}	\left[1-e^{-2\theta_{0}\Delta_{i}}-e^{-2\theta_{0}\Delta_{i}}+e^{-2\theta_{0}\Delta_{i}}     \right]}{\sigma_{01}\sigma_{02}(1+\rho_{0}^{2})^{1/2}(1-e^{-2\theta_{0}\Delta_{i}})}
	\\
	&=&\frac{\sigma_{01}\sigma_{02}\rho_{0}(1-e^{-2\theta_{0}\Delta_{i}})}{\sigma_{01}\sigma_{02}(1+\rho_{0}^{2})^{1/2}(1-e^{-2\theta_{0}\Delta_{i}})}
	\\
	&=&\frac{\rho_0}{(1+\rho_0^2)^{1/2}},
\end{eqnarray*}
Furthermore,
\begin{eqnarray*}	
	\mathbb{E}(Y_{i,n}^{2})&=&\frac{\mathbb{E}\left[(z_{1,i}-e^{-\theta_{0}\Delta_{i}}z_{1,i-1})^{2}
		(z_{2,i}-e^{-\theta_{0}\Delta_{i}}z_{2,i-1})^{2}\right] }{\left[\sigma_{01}\sigma_{02}(1+\rho_{0}^{2})^{1/2}
		(1-e^{-2\theta_{0}\Delta_{i}})\right]^{2}}\\
	&=&\frac{\sigma_{01}^{2}\sigma_{02}^{2}\left[1+2\rho_{0}^{2}+e^{-2\theta_{0}\Delta_{i}}(e^{-2\theta_{0}\Delta_{i}}-2)+2\rho_{0}^{2}e^{-2\theta_{0}\Delta_{i}}(e^{-2\theta_{0}\Delta_{i}}-2)  	 	
		\right]}
	{\left[\sigma_{01}\sigma_{02}(1+\rho_{0}^{2})^{1/2}
		(1-e^{-2\theta_{0}\Delta_{i}})\right]^{2}}\\
	&=&\frac{1+2\rho_{0}^{2}}{1+\rho_{0}^{2}},
\end{eqnarray*}
as is obtained by using Isserlis' theorem for correlated Gaussian random variables. 
Furtermore
\begin{eqnarray*}
	Var(Y_{i,n}) &=& \mathbb{E}(Y_{i,n}^{2})-[\mathbb{E}(Y_{i,n})]^{2}\\
	&=& \frac{1+2\rho_{0}^{2}}{1+\rho_0^2} - \left(\frac{\rho_0}{(1+\rho_0^2)^{1/2}} \right)^{2}\\
	&=& 1. 
\end{eqnarray*}
Hence $\mathbb{E}(|Y_{i,n}| \leq \sqrt{2})$ and so
\begin{eqnarray*}
	\mathbb{E} ( T ) &  \leq & \sqrt{2} 
	\sum_{i=2}^n \left|  
	D_{\theta,i} - \frac{\sigma_{01} \sigma_{02} \sqrt{1+\rho_0^2} \theta_0}{\theta^2}
	\right| \\
	& = &
	\sqrt{2}
	\sigma_{01} \sigma_{02} \sqrt{1+\rho_0^2}
	\sum_{i=2}^n 
	\left|
	\frac{
		(1-e^{-2 \theta_0 \Delta_i}) 2 \Delta_i e^{-2 \theta  \Delta_i}
	}
	{
		(1-e^{-2 \theta \Delta_i})^2
	}
	- \frac{\theta_0}{\theta^2}
	\right|. \\
\end{eqnarray*}
One can show, from a Taylor expansion and since $\theta \in \Theta$ with $\Theta$ compact in $(0,+ \infty)$, that
\[
\sup_{n \in \mathbb{N}, i=2,...,n} 
\sup_{\theta \in \Theta}
\frac{1}{\Delta_i}
\left|
\frac{
	(1-e^{-2 \theta_0 \Delta_i}) 2 \Delta_i e^{-2 \theta  \Delta_i}
}
{
	(1-e^{-2 \theta \Delta_i})^2
}
- \frac{\theta_0}{\theta^2}
\right|
< \infty.
\] 
Hence $\mathbb{E}(T) = O_u(1)$ and $T = 0_{up}(1)$. Hence, finally
\[
G = -\frac{\sigma_{01} \sigma_{02} \sqrt{1+\rho_0^2} \theta_0}{\theta^2}
\sum_{i=2}^n Y_{i,n}
+ O_{up}(1).
\]



 \section*{References}

 \bibliographystyle{elsarticle-num} 
 \bibliography{clasymf}


%\begin{thebibliography}{00}


%\bibitem{}

%\end{thebibliography}
\end{document}